\documentclass[a4paper,11pt]{amsart}
\usepackage[english]{babel}
\usepackage[latin1]{inputenc}
\usepackage[T1]{fontenc}
\usepackage{amssymb,amsfonts,amsmath,amsthm}
\usepackage{paralist}
\usepackage[colorlinks=true, linkcolor=blue, citecolor=red, urlcolor=blue]{hyperref}
\usepackage[stretch=10,shrink=10,step=4,kerning=true,protrusion=true,final]{microtype}
\usepackage[arrow,matrix,cmtip]{xy} 

\newtheorem{theorem}{Theorem}
\newtheorem{proposition}[theorem]{Proposition}

\newtheorem{lemma}[theorem]{Lemma}

\newtheorem{claim}[theorem]{Claim}

\theoremstyle{definition}
\newtheorem{definition}[theorem]{Definition}

\newtheorem{criterion}[theorem]{Criterion}

\theoremstyle{remark}
\newtheorem{remark}[theorem]{Remark}
\newtheorem{remarks}[theorem]{Remarks}
\newtheorem{example}[theorem]{Example}

\newtheoremstyle{theoremwithref}{}{}{\itshape}{}{\bfseries}{.}{.5em}{#1 #2 #3}
\theoremstyle{theoremwithref}

\newcommand{\ie}{i.e.\ }
\newcommand{\eg}{e.g.\ }
\newcommand{\resp}{resp.\ }

\newcommand{\PP}{\mathbf{P}}
\newcommand{\CC}{\mathbf{C}}

\newcommand{\RR}{\mathbf{R}}

\newcommand{\NN}{\mathbf{N}}

\newcommand{\GL}{\mathrm{GL}}
\newcommand{\SL}{\mathrm{SL}}
\newcommand{\PSL}{\mathrm{PSL}}

\newcommand{\OO}{\mathrm{O}}

\newcommand{\PSU}{\mathrm{PSU}}
\newcommand{\U}{\mathrm{U}}

\newcommand{\g}{\mathfrak{g}}

\newcommand{\aaa}{\mathfrak{a}}

\newcommand{\Hom}{\mathrm{Hom}}

\newcommand{\Ker}{\mathrm{Ker}}

\newcommand{\Ad}{\mathrm{Ad}}
\newcommand{\ad}{\mathrm{ad}}

\newcommand{\kl}{\kappa}

\newcommand{\F}{\mathcal{F}}

\newcommand*\overX{%
  \hspace*{0.2em}\vbox{%
    \hrule height 0.3pt%
    \kern0.15ex%
    \hbox{%
      \kern-0.2em%
      \ensuremath{X}%
      \kern-0.1em%
    }%
  }%
}

\newcommand*\overY{%
  \vbox{%
    \hrule height 0.3pt%
    \kern0.15ex%
    \hbox{%
      \kern-0.1em%
      \ensuremath{Y}%
      \kern-0.1em%
    }%
  }%
}

\newcommand*\overZ{%
  \vbox{%
    \hrule height 0.3pt%
    \kern0.15ex%
    \hbox{%
      \kern-0.2em%
      \ensuremath{Z}%
      \kern-0.1em%
    }%
  }%
}

\newcommand{\dist}{\mathrm{dist}}
\newcommand{\haus}{\mathrm{Hdist}}

\title[Tameness of Riemannian locally symmetric spaces]{Tameness of Riemannian locally symmetric spaces arising from Anosov representations}

\author[O.\ Guichard]{Olivier Guichard}
\address{Universit\'e de Strasbourg, IRMA, 7 rue  Descartes, 67000 Strasbourg, France}
\email{olivier.guichard@math.unistra.fr}

\author[F.\ Kassel]{Fanny Kassel}
\address{CNRS and Universit\'e Lille 1, Laboratoire Paul Painlev\'e, 59655 Villeneuve d'Ascq Cedex, France}
\email{fanny.kassel@math.univ-lille1.fr}

\author[A.\ Wienhard]{Anna Wienhard}
\address{Ruprecht-Karls Universit\"at Heidelberg, Mathematisches Institut, Im Neuenheimer Feld~288, 69120 Heidelberg, Germany
\newline HITS gGmbH, Heidelberg Institute for Theoretical Studies, Schloss-Wolfs\-brunnen\-weg 35, 69118 Heidelberg, Germany}
\email{wienhard@uni-heidelberg.de}
\thanks{FK was partially supported by the Agence Nationale de la Recherche under the grant DiscGroup (ANR-11-BS01-013) and through the Labex CEMPI (ANR-11-LABX-0007-01).
AW was partially supported by the National Science Foundation under agreement DMS-1536017, by the Sloan Foundation, by the Deutsche Forschungsgemeinschaft, and by the European Research Council under ERC-Consolidator grant 614733.
Part of this work was carried out while the authors were in residence at the MSRI in Berkeley, California, supported by the National Science Foundation under grant 0932078~000.
The authors also acknowledge support from U.S. National Science Foundation grants DMS 1107452, 1107263, 1107367 ``RNMS: GEometric structures And Representation varieties'' (the GEAR Network).}

\begin{document}

\numberwithin{theorem}{section}
\numberwithin{equation}{section}

\begin{abstract}
We construct compactifications of Riemannian locally symmetric spaces arising as quotients by Anosov representations.
These compactifications are modeled on generalized Satake compactifications and, in certain cases, on maximal Satake compactifications.
We deduce that these Riemannian locally symmetric spaces are topologically tame, \ie homeomorphic to the interior of a compact manifold with boundary.
We also construct domains of discontinuity (not necessarily with a compact quotient) in a much more general setting.
\end{abstract}
\maketitle

\section{Introduction}

Any discrete subgroup $\Lambda$ of a semisimple (or reductive) Lie group $G$ acts properly discontinuously by isometries on the Riemannian symmetric space $X = G/K$ of~$G$.
The quotient space $M= \Lambda\backslash X$ is a Riemannian locally symmetric orbifold, which is noncompact except if $\Lambda$ is a uniform lattice in~$G$.
When $M$ has finite volume (\ie $\Lambda$ is a lattice), compactifications of~$M$ have been well studied: see \cite{Borel_Ji} for an overview of various compactifications with their properties and uses. 
When $M$ has infinite volume, compactifications of~$M$ have been mainly studied in the case that $G$ has real rank one, \ie that $X$ is a negatively curved manifold.
In this case, compactifications of~$M$ have been constructed for geometrically finite representations (see \cite[Prop.\,3.5]{Ji_NovikovConj}, based on \cite[Th.\,6.5]{Apanasov_Xie}). 
Recently there has been a growing interest in Zariski-dense subgroups of semisimple Lie groups, also of higher rank, which are not lattices, \ie for which $M$ has infinite volume.
However, when $G$ has higher real rank and $\Lambda$ has infinite covolume, compactifications of~$M$ are not well studied, and very little is known. 

In this paper we construct compactifications of~$M$ when $\Lambda$ is the image of an Anosov representation.
When $G$ has real rank one, the images of Anosov representations are exactly the convex cocompact subgroups; when $G$ has higher real rank, images of Anosov representations provide a meaningful generalization of convex cocompact subgroups \cite{Labourie_anosov, Guichard_Wienhard_DoD, KapovichLeebPorti, KapovichLeebPorti14, KapovichLeebPorti14_2}.

\begin{theorem} \label{thm:main}
Let $X=G/K$ be a Riemannian symmetric space, where $G$ is a noncompact real semisimple Lie group and $K$ a maximal compact subgroup of~$G$.
Let $\Gamma$ be a word hyperbolic group and $P$ a proper parabolic subgroup of~$G$.
For any $P$-Anosov representation $\rho\in\Hom(\Gamma,G)$, the Riemannian locally symmetric space $\rho(\Gamma)\backslash X$ admits a compactification which is an orbifold with corners, locally modeled on a generalized Satake compactification of~$X$.
\end{theorem}

We introduce generalized Satake compactifications in Appendix~\ref{sec:satake-comp} (Definition~\ref{defi:gen-Satake}).
They provide a natural extension of the class of Satake compactifications, which satisfies the functorial property that the closure of a totally geodesic subsymmetric space $Y\subset X$ in a generalized Satake compactification of~$X$ is a generalized Satake compactification of~$Y$.
This is not true for Satake compactifications.
In Theorem~\ref{thm:main} the generalized Satake compactification \emph{dominates} (\ie admits a continuous $G$-equivariant map to) the maximal Satake compactification of~$X$.

For specific Anosov representations, we can improve Theorem~\ref{thm:main} and construct a compactification modeled on the maximal Satake compactification of~$X$.

\begin{theorem} \label{thm:main_max_Satake}
Let $X=G/K$ be a Riemannian symmetric space where $G$ is a noncompact real semisimple Lie group.
Then there exists a maximal proper parabolic subgroup $P$ of~$G$ such that for any word hyperbolic group~$\Gamma$ and any $P$-Anosov representation $\rho : \Gamma\to G$, the Riemannian locally symmetric space $\rho(\Gamma)\backslash X$ admits a compactification which is an orbifold with corners, locally modeled on the \emph{maximal Satake compactification} of~$X$.
\end{theorem}

For more precise statements, we refer to Theorem~\ref{thm:max_Satake} in the case that $G$ is simple and Theorem~\ref{thm:max-satak-semisimple} in the general case.

\begin{remarks}
\begin{enumerate}
  \item If $P'$ is a parabolic subgroup of~$G$ contained in~$P$, then any $P'$-Anosov representation $\rho: \Gamma \to G$ is $P$-Anosov.
  In particular, Theorem~\ref{thm:main_max_Satake} applies to any $P'$-Anosov representation with $P'\subset P$ (for instance to any $P_{\mathrm{min}}$-Anosov representation where $P_{\mathrm{min}}$ is a minimal parabolic subgroup of~$G$).
  \item For Anosov representations $\rho:\Gamma \to \OO(b)$ (\resp \(\OO(b_\CC)\)) into the orthogonal of a nondegenerate real (\resp complex) symmetric bilinear form \(b\) (\resp \(b_\CC\)), we actually construct compactifications of the Riemannian locally symmetric spaces that are modeled on a minimal Satake compactification: see Theorems \ref{thm:comp_loc_sym_spc_opq} and~\ref{thm:comp_loc_sym_spc_onC} for precise statements.
  \item In the preprint \cite{KapovichLeeb15}, Kapovich and Leeb construct, by a different method, compactifications modeled on the maximal Satake compactification for Riemannian locally symmetric spaces arising from any Anosov representation.
They also prove a converse statement: if a subgroup \(\Lambda\) of \(G\) is uniformly \(\tau_{mod}\)-regular (a uniform version of the notion of \(P_\theta\)-divergence to be found below in Section~\ref{sec:diverg-repr}) and if the locally symmetric space \(\Lambda \backslash X\) admits a compactification modeled on the maximal Satake compactification of~$X$, then the group \(\Lambda\) is word hyperbolic and the inclusion of \(\Lambda\) in \(G\) is \(P_\theta\)-Anosov.
\end{enumerate}
\end{remarks}

The compactifications of $M = \Lambda \backslash X$ that we construct in Theorems \ref{thm:main} and \ref{thm:main_max_Satake} and their refinements (Theorems \ref{thm:comp_loc_sym_spc_opq}, \ref{thm:max_Satake}, \ref{thm:max-satak-semisimple}, and~\ref{thm:comp_loc_sym_spc_onC}) are all obtained by considering a Satake or generalized Satake compactification $\overX$ of~$X$, and removing from it a {\em bad set}~$\mathcal{N}$, determined by the dynamical properties of sequences of elements of~$\Lambda$, such that the action of $\Lambda$ on $\overX \smallsetminus \mathcal{N}$ is properly discontinuous.
The idea of describing $\mathcal{N}$ in terms of dynamics of sequences is inspired by \cite{Frances_Lorentzian}.

In Theorem~\ref{thm:bdf} we define a bad set~$\mathcal{N}$ in a compactification~$\overX$ and obtain a properly discontinuous action on $\overX \smallsetminus \mathcal{N}$ for \emph{any} discrete subgroup $\Lambda$ of~$G$.
This yields a manifold with corners containing \(\Lambda \backslash X\) as a dense subset.

For Anosov representations, the compactification~$\overX$ and the bad set~$\mathcal{N}$ can be chosen in such a way that the quotient $\Lambda\backslash (\overX\smallsetminus\mathcal{N})$ is compact, providing a genuine compactification of $M = \Lambda\backslash X$.
Let us emphasize that the topology on $\overX \smallsetminus \mathcal{N}$ is induced by the inclusion into \(\overX\).
This is in contrast to the situation of Satake compactifications of  Riemannian locally symmetric spaces of finite volume, where one takes the union of $X$ with a subset of~$\overX \smallsetminus X$, but changes the topology on the union. 
A combination of these two strategies might provide an approach to compactify Riemannian locally symmetric spaces of infinite volume that do not arise from Anosov representations, but from more general discrete subgroups.

We apply our construction of compactifications to prove topogical tameness.

\begin{theorem} \label{thm:tame}
Let $X=G/K$ be a Riemannian symmetric space, where $G$ is a noncompact real semisimple Lie group and $K$ a maximal compact subgroup of~$G$.
Let $\Gamma$ be a torsion-free word hyperbolic group and $P$ a proper parabolic subgroup of~$G$.
For any $P$-Anosov representation $\rho\in\Hom(\Gamma,G)$, the Riemannian locally symmetric space $\rho(\Gamma)\backslash G/K$ is topologically tame, \ie homeomorphic to the interior of a compact manifold with boundary.
\end{theorem}

\subsection*{Organization of the paper}
\label{sec:organization-paper}

In Section~\ref{sec:lie-groups} we introduce some notation and recall some basic facts on semisimple Lie groups and their parabolic subgroups.
In Section~\ref{sec:anos-repr} we recall the notions of limit set and Anosov representation, and establish some useful properties.
In Section~\ref{sec:comp-anos-locally} we prove Theorem~\ref{thm:comp_loc_sym_spc_opq}, which gives compactifications modeled on minimal Satake compactifications for orthogonal groups. 
From this, in Section~\ref{sec:comp-riem-locally-gal}, we deduce Theorems~\ref{thm:main} and~\ref{thm:main_max_Satake} in full generality as well as Theorem~\ref{thm:tame}; compactifications for complex orthogonal groups (Theorem~\ref{thm:comp_loc_sym_spc_onC}) are also discussed.
In Appendix~\ref{sec:satake-comp} we give a description of Satake compactifications and a few properties of generalized Satake compactifications.

\subsection*{Acknowledgements}

We are grateful to Lizhen Ji for his interest in this work and for motivating discussions about it.
We thank Misha Kapovich and Bernhard Leeb for pointing out a mistake in a previous version of the paper.

\section{Background on Lie groups and their parabolic subgroups}
\label{sec:lie-groups}

In this section we recall some basic facts on the structure of real reductive Lie groups and their parabolic subgroups. 

Let $G$ be a real reductive Lie group with Lie algebra~$\g$.
In the whole paper, we assume $G$ to be noncompact, equal to a finite union of connected components (for the real topology) of $\mathbf{G}(\RR)$ for some algebraic group~$\mathbf{G}$.

\subsection{Restricted roots}
\label{sec:gener-reduct-lie}

Let $K$ be a maximal compact subgroup of~$G$, with Lie algebra~$\mathfrak{k}$, and let 
\(\aaa\) be a maximal abelian subspace of the orthogonal complement of $\mathfrak{k}$ in~$\g$ for the Killing form \(\kl\).
The \emph{real rank} of~$G$ is by definition the dimension of~$\aaa$.
Let $\Sigma$ be the set of restricted roots of $\aaa$ in~$\g$, \ie the set of nonzero linear forms $\alpha\in\aaa^*$ for which 
\[ \g_{\alpha} := \{ z\in\g ~|~ \ad(a)(z) = \langle{\alpha, a}\rangle\,z \quad\forall a\in\aaa\} \]
is nonzero.
(We denote by $\langle\cdot,\cdot\rangle : \aaa^*\times\aaa\to\RR$ the natural pairing.)
Let $\Delta\subset\Sigma$ be a system of \emph{simple restricted roots}, \ie any element of~$\Sigma$ is expressed uniquely as a linear combination of elements of \(\Delta\) with coefficients all of the same sign.
Let 
\[ \overline{\aaa}^+ := \{Y \in \aaa \mid \langle{\alpha,Y}\rangle \geq 0\quad \forall \alpha \in \Delta\} \]
be the closed positive Weyl chamber of~$\aaa$ associated with~$\Delta$.
The \emph{restricted Weyl group} of $\aaa$ in~$\g$ is the group $W=N_K(\aaa)/Z_K(\aaa)$, where $N_K(\aaa)$ (\resp $Z_K(\aaa)$) is the normalizer (\resp centralizer) of $\aaa$ in~$K$.
There is a unique element $w_0 \in W$ such that $w_0\cdot (-\overline{\aaa}^+)=\overline{\aaa}^+$; the involution of~$\aaa$ defined by $Y\mapsto -w_0\cdot Y$ is called the \emph{opposition involution}.
The corresponding dual linear map preserves~$\Delta$; we shall denote it by
\begin{align}\label{eqn:opp-inv}
  \aaa^{\ast} & \longrightarrow  \aaa^{\ast}\\
  \alpha\, & \longmapsto  \alpha^{\star} = -w_0\cdot \alpha. \notag
\end{align}

\subsection{Cartan decomposition}
\label{sec:cartan-decomposition}

Recall that $G$ admits the \emph{Cartan decomposition} $G=K \exp(\overline{\aaa}^+)K$: any $g\in G$ may be written
\begin{equation} \label{eqn:g=kak'}
g = k_g a_g \ell_g
\end{equation}
for some $k_g,\ell_g\in K$ and a unique \(a_g\) in \(\exp (\overline{\aaa}^+)\), $\mu(g)=\log a_g$ is called the Cartan projection of \(g\) (see \cite[Ch.\,IX, Th.\,1.1]{Helgason}).
This defines a proper, continuous, surjective map
\[ \mu :\ G \longrightarrow  \overline{\aaa}^+ \]
called the \emph{Cartan projection}, inducing a homeomorphism $K \backslash G / K \simeq \overline{\aaa}^+$.
The pair \((k_g,\ell_g)\) is not unique, but is determined uniquely up to the action of the centralizer of \(\mu(g)\) in~\(K\).

\subsection{Parabolic subgroups, flag varieties and transversality}
\label{sec:parabolic-subgroups}

Let $\Sigma^+ \subset \Sigma$ be the set of \emph{positive restricted roots} with respect to~$\Delta$, \ie restricted roots that are nonnegative linear combinations of elements of~\(\Delta\).
For any nonempty subset $\theta\subset\Delta$, we denote by $P_\theta$ the normalizer in~$G$ of the Lie algebra $\mathfrak{u}_\theta = \bigoplus_{\alpha \in \Sigma^+ \smallsetminus  \mathrm{span}(\Delta \smallsetminus \theta)} \g_\alpha$. 
Explicitly,
\[\mathrm{Lie}(P_{\theta})= \mathfrak{p}_\theta = \g_0 \oplus \bigoplus_{\alpha\in\Sigma^+} \g_{\alpha} \oplus \bigoplus_{\alpha\in\Sigma^+\cap\mathrm{span}(\Delta\smallsetminus\theta)} \g_{-\alpha}.\]
In particular, $P_\emptyset = G$ and $P_\Delta$ is a minimal parabolic subgroup of~$G$.\footnote{This is the same convention as in \cite{GGKW_anosov, GGKW_compact}, but the opposite convention to \cite{Guichard_Wienhard_DoD}.}
Any parabolic subgroup of~$G$ is conjugate to~$P_{\theta}$ for some $\theta\subset\Delta$.

The standard opposite parabolic subgroup to $P_{\theta}$ is the normalizer \(P_{\theta}^{-}\) of \( \mathfrak{u}_{\theta}^{-} = \bigoplus_{\alpha \in \Sigma^+ \smallsetminus  \mathrm{span}(\Delta \smallsetminus \theta)} \g_{-\alpha}.\)
Note that \(P_{\theta}^{-}\) is conjugate to~\(P_{\theta^\star}\).
We shall consider the flag varieties
\begin{align*}
  \mathcal{F}_\theta &= \{ P \subset G \mid P \text{ is conjugate to }
  P_\theta\} \simeq G/P_\theta,\\
  \mathcal{F}_{\theta^\star} &= \{ P \subset G \mid P \text{ is conjugate to }
  P_{\theta}^{-}\} \simeq G/P_{\theta}^{-} \simeq G/P_{\theta^\star}.
\end{align*}
  
\begin{definition}\label{defi:para_transverse}
  A pair \((P,Q) \in \mathcal{F}_\theta \times \mathcal{F}_{\theta^\star}\) of parabolic subgroups of~$G$ is called \emph{transverse} if \(P \cap Q\) is a reductive Lie group, or equivalently if \((P,Q)\) is conjugate to \((P_\theta, P_{\theta}^{-})\) under the diagonal action of~$G$.
\end{definition}

\subsection{Example: general linear groups}
\label{sec:gener-line-groups}

Let \(G=\GL_\RR(V)\), where \(V\) is a real vector space of dimension~\(n\).
We may fix a basis \((e_1, \dots, e_n)\) of~\(V\) and take \(K\) to be \(\OO(n)\) and $\aaa$ to be the space of diagonal matrices in that basis:
\begin{equation*}
  \aaa = \{ \mathrm{diag}( \lambda_1, \dots, \lambda_n) \mid
  \lambda_1, \dots, \lambda_n \in \RR\}.
\end{equation*}
Let \((\varepsilon_1, \dots, \varepsilon_n)\) be the standard basis of~\(\aaa^*\), \ie \(\langle{\varepsilon_i, \mathrm{diag}(
\lambda_1, \dots, \lambda_n)}\rangle = \lambda_i\) for all~$i$.
The root system is
\[ \Sigma = \{ \varepsilon_i - \varepsilon_j \mid i\neq j,\ 1\leq i,j\leq n\}. \]
A system of simple roots is
\[ \Delta = \{ \alpha_i \mid 1\leq i\leq n-1\}, \]
where $\alpha_i := \varepsilon_i - \varepsilon_{i+1}$.
The opposition involution switches \(\alpha_i\) and~\(\alpha_{n-i}\).

The parabolic subgroup \(P_{\{\alpha_i\}}\) will be denoted~\(P_i\); it is the stabilizer in \(\GL_\RR(V)\) of the subspace \(\RR e_1 \oplus \cdots \oplus \RR e_i\) of~$V$.
The flag variety \(\mathcal{F}_{\{\alpha_i\}} = \GL_\RR(V)/P_i\) identifies with the Grassmannian \(\mathrm{Gr}_i(V) \simeq \mathrm{Gr}_{n-i}(V^{\ast})\).
In particular, \(\mathcal{F}_{\{\alpha_1\}}\) identifies with the projective space \(\PP(V)\) and \(\mathcal{F}_{\{\alpha_{n-1}\}}\) with the projective dual space \(\PP(V^{\ast})\).
The notion of transversality on \(\mathrm{Gr}_i(V) \times \mathrm{Gr}_{n-i}(V)\) from Definition~\ref{defi:para_transverse} is the natural one: a pair \((W_i,W_{n-i})\) is transverse if and only if \(W_i \oplus W_{n-i} = V\).

\subsection{Example: indefinite orthogonal groups}
\label{sec:orthogonal-groups}

Let \(b\) be a nondegenerate bilinear symmetric form of signature \((p,q)\) on a real vector space~\(V\).
Suppose that \((p,q)\neq (1,1)\) and that \(p\geq q>0\) (the case \(q\geq p>0\) is similar).
Let $G$ be the orthogonal group \(\OO(b)\).
There is a basis \((e_1, \dots, e_{p+q})\) of \(V\) such that for any \(x= \sum_{i=1}^{p+q} x_i e_i\) and \(y= \sum_{i=1}^{p+q} y_i e_i\),
\begin{equation*}
  b(x,y) = \sum_{i=1}^{q} \bigl( x_i y_{p+q-i+1} +x_{p+q-i+1} y_{i}\bigr) + \sum_{i=q+1}^{p} x_{i} y_i.
\end{equation*}
We may take \(K = \OO(p+q) \cap G\), which is isomorphic to \(\OO(p) \times \OO(q)\), and
\begin{equation*}
  \aaa = \{ \mathrm{diag}(\lambda_1 , \dots, \lambda_q, 0, \dots, 0, -\lambda_q, \dots, -\lambda_1) \mid \lambda_1, \dots, \lambda_q \in \RR\}.
\end{equation*}
Let \((\varepsilon_1, \dots, \varepsilon_q)\) be the standard basis of~\(\aaa^*\), \ie
\[ \langle\varepsilon_i, \mathrm{diag}(\lambda_1 , \dots, \lambda_q, 0, \dots, 0, -\lambda_q, \dots, -\lambda_1)\rangle = \lambda_i \]
for all~$i$.
The restricted root system is 
\begin{align*}
  \Sigma &= \{ \pm \varepsilon_i \pm \varepsilon_j \mid 1\leq i< j \leq q\} \cup \{
           \pm \varepsilon_i \mid 1\leq i \leq q \} \text{ if } p>q \quad \text{(type }B_q\text{)},\\
  \Sigma &= \{ \pm\varepsilon_i \pm \varepsilon_j \mid 1\leq i<j \leq p\} \text{
           if } p=q \quad (\text{type } D_p).
\end{align*}
A system of simple restricted roots is \(\Delta = \{ \alpha_1,\dots,\alpha_q\}\) where $\alpha_i = \varepsilon_i - \varepsilon_{i+1}$ for $1\leq i\leq q-1$ and
\begin{equation} \label{eq:simple_roots_B_q}
\alpha_q = \left\{ \begin{array}{ll}
           \varepsilon_q & \text{ if } p>q,\\
           \varepsilon_{q-1}+\varepsilon_q & \text{ if } p=q.
           \end{array}\right.
\end{equation}
The opposition involution fixes the simple root \(\alpha_1=\varepsilon_1-\varepsilon_2\).
 The parabolic subgroup \(P_{\{\alpha_1\}} = P_{\{\alpha_1\}^\star}\) will be denoted \(P_1(b)\); it is the stabilizer in \(\OO(b)\) of the line \(\RR e_1\).
The opposite parabolic subgroup \(P^{-}_{\{\alpha_1\}}\) is the stabilizer of \(\RR e_{p+q}\).
The flag variety $\F_{\{\alpha_1\}} = \OO(b)/P_1(b)$ identifies with the space of $b$-isotropic lines in~$V$ (a closed subset of $\PP(V)$) and will be denoted \(\mathcal{F}_1(b)\).
A pair \((\ell,\ell')\) of elements of \(\mathcal{F}_1(b)\) is transverse if and only if \(\ell^{\perp_{b}} + \ell' = V\).

Suppose \(p>q\).
For \(1\leq i\leq q\), the parabolic subgroup \(P_{\{\alpha_i\}}\) will be denoted \(P_i( b)\); it is the stabilizer in $\OO(b)$ of \(\RR e_1 \oplus \cdots \oplus \RR e_i\).
It is conjugate to its opposite, \(P^{-}_{\{\alpha_i\}}\), which is the stabilizer of \(\RR e_{p+q-i+1} \oplus \cdots \oplus \RR e_{p+q}\). 
The flag variety \(\mathcal{F}_i(b) = \OO(b)/P_i(b)\) is the space of \(b\)-isotropic \(i\)-dimensional subspaces of~\(V\). 
A pair \((W,W')\) in \(\mathcal{F}_i(b)\) is transverse in the sense of Definition~\ref{defi:para_transverse} if and only if \(W^{\perp_{b}}+ W'=V\).

Suppose \(p=q\).
For \(1\leq i\leq p-1\) we denote again by \(P_i(b)\) the stabilizer in \(\OO(b)\) of the $b$-isotropic $i$-dimensional subspace \(\RR e_1 \oplus \cdots \oplus \RR e_i\).
For \(i<p-1\), \(P_i(b)\) is \(P_{\{\alpha_i \}}\) and \(P_{p-1}(b)\) is \(P_{\{\alpha_{p-1}, \alpha_p\}}\). For any \(i<p\), \(P_i(b)\) is conjugate to its opposite.
The corresponding homogenous space \(\mathcal{F}_i(b)\) is the space of \(b\)-isotropic \(i\)-planes of \(V\).
Transversality is as above. 
The parabolic subgroups \(P_{\{\alpha_{p-1}\}}\) and \(P_{\{\alpha_{p}\}}\) can be viewed as stabilizers of isotropic \(p\)-planes; they are always conjugate under an element of \(\OO(b)\).
The opposition involution fixes \(\alpha_{p-1}\) and \(\alpha_p\) if \(p\) is even and exchanges them if \(p\) is odd.

\section{Divergence and Anosov representations}
\label{sec:anos-repr}

In this section we recall the notion of limit set in a broad setting, as well as the definition of Anosov representations, and establish some properties of Anosov representations that will be used later in the paper.
We continue with the notation of Section~\ref{sec:lie-groups}.

\subsection{$P_{\theta}$-divergence and limit sets}
\label{sec:diverg-repr}

Let $\theta\subset\Delta$ be a nonempty subset of the simple restricted roots of the noncompact real reductive Lie group~$G$.
As in \cite[\S\,5]{GGKW_anosov}, we define a map $\Xi_{\theta} : G\to\F_{\theta}$ as follows: for any $g\in G$, we choose $k_g, \ell_g \in K$ such that $g =  k_g \exp(\mu(g)) \ell_g$, and set
\begin{equation} \label{eqn:Xi-theta+-}
\Xi_{\theta}(g) = k_g \cdot P_{\theta} \in \F_{\theta} = G/P_{\theta}.
\end{equation}
This does not depend on the choice of $k_g, \ell_g$ as soon as $\langle\alpha,\mu(g)\rangle > 0$ for all $\alpha\in\theta$ (see \cite[Ch.\,IX, Cor.\,1.2]{Helgason}).
We adopt the following terminology.

\begin{definition}\label{defi:theta_divergent-seq}
A sequence $(g_n)\in G^{\NN}$ is \emph{$P_{\theta}$-divergent} if for any \(\alpha\in \theta\),
   \[ \lim_{n \to +\infty} \langle \alpha, \mu(g_n) \rangle = +\infty. \]
\end{definition}

\begin{definition}\label{def:lim-set}
Let $\Lambda$ be a discrete subgroup of~$G$.
Let $\theta\subset \Delta$ be a nonempty subset such that $\Lambda$ admits a $P_{\theta}$-divergent sequence.
The \emph{limit set} \(\mathcal{L}^{\mathcal{F}_\theta}_{\Lambda}\) of $\Lambda$ in $\F_\theta$ is the set of all limits in~$\F_{\theta}$ of sequences $(\Xi_{\theta}(\gamma_n))_{n\in\NN}$ where $(\gamma_n)\in\Lambda^{\NN}$ is $P_{\theta}$-divergent.
\end{definition}

By \cite[\S\,3.2]{Benoist}, if $\Lambda$ is Zariski-dense in~$G$, then $\Lambda$ contains \emph{$\theta$-proximal} elements, \ie elements with a unique attracting fixed point in~$\F_{\theta}$, and \(\mathcal{L}^{\mathcal{F}_\theta}_{\Lambda}\) is the closure of the set of attracting fixed points of these elements.

\begin{definition}\label{defi:theta_divergent-gp}
Let $\Gamma$ be a discrete group.
A representation \(\rho: \Gamma \to\nolinebreak G\) is \emph{\(P_\theta\)-divergent} if all sequences of pairwise distinct elements in $\rho(\Gamma)$ are~$P_{\theta}$-diver\-gent; equivalently, for any $\alpha\in\theta$, \(\lim_{\gamma \to \infty} \langle \alpha, \mu( \rho (\gamma))\rangle = +\infty\), \ie for any \(M>0\) the set \(\{\gamma\in \Gamma \mid\nolinebreak\langle\alpha,\mu(\rho(\gamma))\rangle <\nolinebreak M\}\) is finite.
\end{definition}

If $\rho : \Gamma\to G$ is $P_\theta$-divergent, then it has finite kernel and discrete image. 

\begin{remarks} \label{rem:div}
\begin{enumerate}
  \item\label{item:defitheta} A particular case of Definition~\ref{defi:theta_divergent-gp} was
  used in \cite[\S\,7.2]{Guichard_Wienhard_DoD}. 
  The definition is equivalent to the notion of \emph{weakly \(\tau_{mod}\)-regular} subgroup of \cite[Def.\,5.6]{KapovichLeebPorti14} where $\tau_{mod}$ is the facet $\overline{\aaa}^+ \cap \bigcap_{\alpha\in\Delta\smallsetminus\theta} \Ker(\alpha)$ of~$\overline{\aaa}^+$.
  \item \label{item:div-star} The equality $\langle\alpha,\mu(g)\rangle = \langle\alpha^{\star},\mu(g^{-1})\rangle$ for all $\alpha\in\Delta$ and $g\in G$ implies that a representation \(\rho: \Gamma \to G\) is \(P_\theta\)-divergent if and only if it is \(P_{\theta\cup\theta^{\star}}\)-divergent.
  \item\label{item:rem-div-Xi-g-inv} If \(g = k_g \exp(\mu(g)) \ell_g\) with $k_g,\ell_g\in K$, then \(\Xi_\theta (g^{-1}) = \ell_{g}^{-1} \cdot P_{\theta^\star}^{-}\) and \(\Xi_\theta(g^{-1})\) does not depend on the choices as soon as \(\langle \alpha, \mu(g)\rangle >0\) for all \(\alpha\in \theta^\star\).
Therefore, if a sequence \((\gamma_n)_{n\in \NN}\) of \(\Lambda^\NN\) is \(P_{\theta^\star}\)-divergent and if the sequence \(( \Xi_\theta( \gamma_{n}^{-1}))_{n\in \NN}\) converges, then its limit belongs to the limit set~\(\mathcal{L}^{\mathcal{F}_\theta}_{\Lambda}\).
  \end{enumerate}
\end{remarks}

\subsection{Anosov representations}

We now suppose that \(\Gamma\) is word hyperbolic and denote by \(\partial_\infty \Gamma\) its boundary at infinity. 
The following definition of Anosov representations is not the original one from \cite{Labourie_anosov, Guichard_Wienhard_DoD}, but an equivalent one taken from \cite{GGKW_anosov}.

\begin{definition} \label{defi:ano_theta}
Let $\Gamma$ be a word hyperbolic group.
  A representation \(\rho: \Gamma \to G\) is \emph{\(P_\theta\)-Anosov} if it is $P_\theta$-divergent and there exist continuous, \(\rho\)-equivariant maps \(\xi^+ : \partial_\infty \Gamma\to\mathcal{F}_\theta\) and \(\xi^- : \partial_\infty \Gamma\to
  \mathcal{F}_{\theta^\star}\) that are transverse and dynamics-preserving.
\end{definition}

By \emph{dynamics-preserving} we mean that if $\eta$ is the attracting fixed point of some element $\gamma\in\Gamma$ in $\partial_{\infty}\Gamma$, then $\xi^+(\eta)$ (\resp $\xi^-(\eta)$) is an attracting fixed point of $\rho(\gamma)$ in $\mathcal{F}_{\theta}$ (\resp $\mathcal{F}_{\theta^{\star}}$).
By \emph{transverse} we mean that pairs of distinct points in \(\partial_\infty \Gamma\) are sent to transverse pairs in $\mathcal{F}_\theta\times\mathcal{F}_{\theta^{\star}}$ (Definition~\ref{defi:para_transverse}).

The maps $\xi^+$ and~$\xi^-$ are unique, entirely determined by~$\rho$.
The set of $P_{\theta}$-Anosov representations is open in $\Hom(\Gamma,G)$ \cite{Labourie_anosov, Guichard_Wienhard_DoD}.

\begin{remarks} \label{rem:theta*Ano}
\begin{asparaenum}
  \item\label{item:star} By Remark~\ref{rem:div}.\eqref{item:div-star} (see also \cite[Lem.\,3.18]{Guichard_Wienhard_DoD}), the representation \(\rho: \Gamma \to G\) is \(P_\theta\)-Anosov if and only if it is \(P_{\theta\cup\theta^{\star}}\)-Anosov.
  \item \label{item:xi+=xi-} When \(\theta=\theta^\star\), the two flags varieties \(\mathcal{F}_\theta\) and~\(\mathcal{F}_{\theta^\star}\) coincide and the two boundary maps \(\xi^+\) and~\(\xi^-\) of a $P_{\theta}$-Anosov representation are equal.
\end{asparaenum}
\end{remarks}

\begin{example} \label{ex:Ano-GL}
Let $G=\GL_\RR(V)$ and $\theta=\{\alpha_i\}=\{\varepsilon_i-\varepsilon_{i+1}\}$.
The boundary maps of a \(P_i\)-Anosov representation \(\rho:\Gamma\to G\) are a pair of continuous maps
\begin{equation*}
  \xi_i=\xi^+: \partial_\infty\Gamma \longrightarrow \mathrm{Gr}_i(V) \quad
  \text{and} \quad \xi_{n-i}=\xi^-: \partial_\infty\Gamma \longrightarrow \mathrm{Gr}_{n-i}(V)
\end{equation*}
such that \(\xi_i(\eta)\oplus\xi_{n-i}(\eta')=V\) for all $\eta\neq\eta'$ in \(\partial_\infty \Gamma\), and such that for any \(\gamma\in\Gamma\) with attracting fixed point \(\eta\) in \(\partial_\infty \Gamma\), the element \(\rho(\gamma)\) has attracting fixed points \(\xi_i(\eta)\) in \(\mathrm{Gr}_i(V)\) and \(\xi_{n-i}(\eta)\) in \(\mathrm{Gr}_{n-i}(V)\). 
Here $P_\theta$-divergence means
\begin{equation*}
  \lim_{\gamma\to \infty} \langle \varepsilon_i - \varepsilon_{i+1}, \mu(\rho(\gamma))\rangle = +\infty.
\end{equation*}
\end{example}

\begin{example} \label{ex:Ano-Opq}
Let $G=\OO(b)$ be the orthogonal group of a symmetric bilinear form of signature \((p,q)\) on a real vector space~\(V\), where $p,q\in\NN^*$ and $(p,q)\neq (1,1)$.
Let $\theta=\{ \alpha_i\}=\theta^\star$ where $1\leq i\leq\min(p,q)$; if $p=q$, we assume $i\neq p$ and \(\theta = \{\alpha_{p-1}, \alpha_p\}\) when \(i=p-1\) (see \eqref{eq:simple_roots_B_q}). 
By Remark~\ref{rem:theta*Ano}.\eqref{item:xi+=xi-}, for a \(P_i(b)\)-Anosov representation \(\rho: \Gamma\to G\) there is just one continuous \(\rho\)-equivariant boundary map \(\xi:\partial_\infty \Gamma\to \mathcal{F}_i(b)\).
It is dynamics-preserving and satisfies \(\xi(\eta)^{\perp_{b}} \oplus \xi(\eta')=V\) for all \(\eta\neq\eta'\) in \(\partial_\infty\Gamma\).
Here $P_\theta$-divergence means
\begin{equation*}
  \lim_{\gamma\to \infty} \langle \alpha_i, \mu(\rho(\gamma))\rangle = +\infty.
\end{equation*}
If \(p= q\) and \(i = p-1\), then the limit \(\lim_{\gamma\to \infty} \langle \alpha_{p} , \mu(\rho(\gamma))\rangle =  +\infty\) is also part of $P_\theta$-divergence.
\end{example}

For general $G$ and~$\theta$, we shall use the following description of the limit set.

\begin{lemma}\cite[Th.\,5.2]{GGKW_anosov}
  \label{lem:anosov-lim-set}
  If a representation \(\rho: \Gamma \to G\) is \(P_{\theta}\)-Anosov with boundary map \(\xi^+: \partial_\infty \Gamma \to \mathcal{F}_\theta\), then \(\mathcal{L}^{\mathcal{F}_\theta}_{\rho(\Gamma)} = \xi^+( \partial_\infty \Gamma) \).
\end{lemma}

\subsection{$\theta$-compatibility} \label{sec:anos-repr-orth}

We shall use the following terminology from \cite{GGKW_anosov}.

\begin{definition} \label{defi:theta-compatible}
Let $V$ be a finite-dimensional real vector space and $\theta\subset\nolinebreak\Delta$ a nonempty subset of the simple restricted roots of~$G$.
An irreducible representation $\tau : G\to\GL_{\RR}(V)$ with highest weight~$\chi_{\tau}$ is \emph{$\theta$-compatible} if
\[ \{\alpha \in \Delta \mid (\chi_\tau, \alpha) > 0\} = \theta. \]
\end{definition}

The following proposition was proved in \cite{GGKW_anosov} for $i=1$.

\begin{proposition}\label{prop:theta-comp-Anosov}
Let $(\tau,V)$ be an irreducible, $\theta$-compatible linear representation of~$G$ over~$\RR$.
Let $V^{\chi_{\tau}}$ be the weight space corresponding to the highest weight, let \(i=:\dim_\RR(V^{\chi_\tau}) < n=:\dim_\RR(V)\), and let \(V_{< \chi_\tau}\) be the sum of all the other weight spaces of~$\tau$.
\begin{asparaenum}
\item For any discrete group \(\Gamma\) and any representation \(\rho: \Gamma \to G\),
 \begin{center}
  \(\rho:\Gamma\to G\) is \(P_{\theta}\)-divergent
  \(\Longleftrightarrow\) \(\tau \circ \rho:\Gamma\to \GL_{\RR}(V)\) is \(P_i\)-divergent.
  \end{center}
  \item For any word hyperbolic group \(\Gamma\) and any representation \(\rho: \Gamma \to G\),
 \begin{center}
  \(\rho:\Gamma\to G\) is \(P_{\theta}\)-Anosov
  \(\Longleftrightarrow\) \(\tau \circ \rho:\Gamma\to \GL_{\RR}(V)\) is \(P_i\)-Anosov.
  \end{center}  
  In this case, the boundary maps \(\xi^+ : \partial_{\infty}\Gamma\to G/P_{\theta}\) and \(\xi^- : \partial_{\infty}\Gamma\to G/P_{\theta}^-\) of~\(\rho\) and the boundary maps $\xi_i : \partial_{\infty}\Gamma\to \mathrm{Gr}_i(V)$ and $\xi_{n-i} : \partial_{\infty}\Gamma\to\mathrm{Gr}_{n-i}(V)$ of $\tau\circ\rho$ are related as follows:
for any $\eta\in\partial_{\infty}\Gamma$, if $(\xi^+(\eta), \xi^-(\eta)) = (gP_{\theta}, gP_{\theta}^-)$ where $g\in G$, then $(\xi_i(\eta), \xi_{n-i}(\eta)) = (\tau(g)V^{\chi_{\tau}}, \tau(g)V_{< \chi_{\tau}})$.
\end{asparaenum}
\end{proposition}

\begin{proof}
Identical to the proof of \cite[Prop.\,4.6 \& 4.8]{GGKW_anosov}: one just needs to replace \cite[Lem.\,4.10.(3)]{GGKW_anosov} with the fact (following from \cite[Lem.\,4.10.(1)\&(2)]{GGKW_anosov}) that for any $g\in G$,
\[ \langle \alpha_i, \mu_{\GL_{\RR}(V)}(\tau(g)) \rangle = \min_{\alpha\in\theta} \langle \alpha, \mu_G(g) \rangle. \qedhere\]
\end{proof}

The following result is an easy consequence of Proposition~\ref{prop:theta-comp-Anosov} with $i=1$; we shall use it to reduce to the group $\OO(b)$ in the proof of Theorem~\ref{thm:main}.

\begin{proposition}[{\cite[Lem.\,4.10, Prop.\,6.7 \& Rem.\,6.9]{GGKW_anosov}}] \label{prop:PthetaG_P1opq}
  Let \(\theta\subset\Delta\) be a nonempty subset of the simple restricted roots of~$G$.
  Then there exist a nondegenerate  symmetric bilinear form \(b\) on a real vector space~$V$ and a homomorphism \(\tau: G\to\OO(b)\) with the following properties:
  \begin{asparaenum}
\item For any discrete group \(\Gamma\) and any representation \(\rho: \Gamma \to G\),
 \begin{center}
  \(\rho:\Gamma\to G\) is \(P_{\theta}\)-divergent
  \(\Longleftrightarrow\) \(\tau \circ \rho:\Gamma\to \OO(b)\) is \(P_1(b)\)-divergent.
  \end{center}
  \item For any word hyperbolic group \(\Gamma\) and any representation \(\rho: \Gamma \to G\),
 \begin{center}
  \(\rho:\Gamma\to G\) is \(P_{\theta}\)-Anosov
  \(\Longleftrightarrow\) \(\tau \circ \rho:\Gamma\to \OO(b)\) is \(P_1(b)\)-Anosov.
  \end{center}  
  \end{asparaenum}
\end{proposition}

There are infinitely many such triples $(p,q,\tau)$, see \cite{GGKW_anosov}.

\begin{lemma} \label{lem:theta-comp-GL-Opq}
  Let \(b\) be a nondegenerate symmetric bilinear form of signature \((p,q)\) on a real vector space \(V\), where $p,q\in\NN^*$ and $(p,q)\neq (1,1)$.
  Let \(1 \leq i \leq \min(p,q)\), with \(i<p\) if \(p=q\).
  Let \(\iota : \OO(b) \hookrightarrow \GL_\RR(V)\) be the natural inclusion.
\begin{asparaenum}
\item For any discrete group \(\Gamma\) and any representation \(\rho: \Gamma \to \OO(b)\), 
 \begin{center}
  \(\rho:\Gamma\to\OO(b)\) is \(P_i(b)\)-divergent
  \(\Longleftrightarrow\) \(\tau \circ \rho:\Gamma\to \GL_{\RR}(V)\) is \(P_i\)-divergent.
  \end{center}
  \item \label{item:theta-comp-GL-Opq-Ano} For any word hyperbolic group \(\Gamma\) and any representation \(\rho: \Gamma \to \OO(b)\),
 \begin{center}
  \(\rho:\Gamma\to\OO(b)\) is \(P_i(b)\)-Anosov
  \(\Longleftrightarrow\) \(\tau \circ \rho:\Gamma\to \GL_{\RR}(V)\) is \(P_i\)-Anosov.
  \end{center}  
\end{asparaenum}
\end{lemma}

\begin{proof}
  The action of \(\OO(b)\) on the exterior product \(\bigwedge^i V\) is irreducible and \(\alpha_i\)-compatible, and the highest weight space has dimension~\(1\).
  By Proposition~\ref{prop:theta-comp-Anosov}, the representation \(\rho\) is \(P_i(b)\)-divergent (\resp \(P_i(b)\)-Anosov) if and only if \(\bigwedge^i \rho: \Gamma \to \GL_\RR( \bigwedge^i V)\) is \(P_1\)-divergent (\resp \(P_1\)-Anosov).
The same proposition, applied to the linear representation \(\GL_\RR(V)\to \GL_\RR( \bigwedge^i V)\), implies that \(\iota \circ \rho\) is \(P_i\)-divergent (\resp \(P_i\)-Anosov) if and only if \(\bigwedge^i \rho: \Gamma \to \GL_\RR( \bigwedge^i V)\) is \(P_1\)-divergent (\resp \(P_1\)-Anosov).
The lemma follows.
\end{proof}

\subsection{The adjoint representation} \label{sec:adjoint}

For a noncompact semisimple Lie group~$G$, recall that the Killing form $\kl$ of the Lie algebra~$\g$ is a nondegenerate indefinite symmetric bilinear form on~$\g$.
Let $\Ad: G \to \OO(\kl)\subset\GL_{\RR}(\g)$ be the adjoint representation.
The highest restricted weight $\chi_G\in\Sigma^+$ of $\Ad$ is called the \emph{highest restricted root}.
In the case that $G$ is simple, we prove the following.

\begin{proposition}\label{prop:G_makeP1}
Let $G$ be a real simple Lie group.

\begin{asparaenum}
  \item \label{item:alphaG} There exists a simple restricted root \(\alpha_G \in \Delta\) such that $\Ad : G\to\GL_{\RR}(\g)$ is $\{ \alpha_G, \alpha_G^{\star}\}$-compatible (Definition~\ref{defi:theta-compatible}), \ie
  \[ \{ \alpha \in \Delta \mid ( \chi_G, \alpha) > 0\} = \{ \alpha_G, \alpha_G^{\star}\} \]
where $\chi_G\in\Sigma^+$ is the highest restricted root.
Moreover, \(\alpha_G = \alpha_G^\star\) unless the restricted root system $\Sigma$ is of type \(A_n\).

Let \(d\) be the real dimension of the root space \(\mathfrak{g}_{\chi_G}\).
  \item \label{item:div-alphaG} For any discrete group \(\Gamma\) and any representation \(\rho: \Gamma \to G\),
 \begin{center}
  \(\rho:\Gamma\to G\) is \(P_{\{\alpha_G\}}\)-divergent
  \(\Longleftrightarrow\) \(\Ad\circ \rho:\Gamma\to \OO(\kl)\) is \(P_d(\kl)\)-divergent.
  \end{center}
  \item \label{item:Ano-alphaG} For any word hyperbolic group \(\Gamma\) and any representation \(\rho: \Gamma \to G\),
 \begin{center}
  \(\rho:\Gamma\to G\) is \(P_{\{\alpha_G\}}\)-Anosov
  \(\Longleftrightarrow\) \(\Ad\circ \rho:\Gamma\to \OO(\kl)\) is \(P_d(\kl)\)-Anosov.
  \end{center}
  In this case the boundary map \(\xi: \partial_\infty \Gamma\to G/ P_{\{ \alpha_G, \alpha_{G}^{\star}\}} \) of~$\rho$ and the boundary map \(\xi_d: \partial_\infty \Gamma \to \mathcal{F}_d(\kl)\) of \(\Ad \circ \rho\) are related as follows: for any \(\eta\in\partial_\infty \Gamma\), if \(\xi(\eta)= g \cdot P_{\{ \alpha_G, \alpha_{G}^{\star}\}}\) where $g\in G$, then \(\xi_d(\eta) = \Ad(g)\cdot \mathfrak{g}_{\chi_G}\).
\end{asparaenum}
\end{proposition}

Table~\ref{table3} gives the simple root \(\alpha_G\) and the highest weight \(\chi_G\) for the various restricted root systems, see \cite[Ch.\,X, Th.\,3.28]{Helgason}.

\begin{table}[h!]
\centering
\begin{tabular}{|c|c|c|}
\hline
Type & $\alpha_G$ & $\chi_G$ \tabularnewline
\hline
$A_n$ & $\alpha_1 $ & $\varepsilon_1 - \varepsilon_{n+1} = \alpha_1 + \dots + \alpha_n$\tabularnewline
$B_n$ & $\alpha_2 $& $\varepsilon_1 + \varepsilon_2 = \alpha_1 + 2 \alpha_2 + \dots + 2 \alpha_n$\tabularnewline
$C_n$ &$\alpha_1$  &$2 \varepsilon_1 = 2\alpha_1 + \dots + 2 \alpha_{n-1} + \alpha_n$ \tabularnewline
$BC_n$ & $\alpha_1$ & $2\varepsilon_1 = 2 \alpha_1 + \dots + 2 \alpha_n$ \tabularnewline
$D_n$ & $\alpha_2$ & $\varepsilon_1 + \varepsilon_2 = \alpha_1 + 2 \alpha_2 + \dots + 2 \alpha_{n-2} + \alpha_{n-1} + \alpha_n$ \tabularnewline
$E_6$ & $\alpha_4$& $\alpha_1 + 2 \alpha_2 + 2 \alpha_3 + 3 \alpha_4 + 2 \alpha_5 + \alpha_6$ \tabularnewline
$E_7$ & $\alpha_6$  & $2 \alpha_1 + 2 \alpha_2 + 3 \alpha_3 + 4 \alpha_4 + 3 \alpha_5 + 2\alpha_6 + \alpha_7$\tabularnewline
$E_8$ & $\alpha_7$ & $ 2\alpha_1 + 3 \alpha_2 +  4 \alpha_3 + 6 \alpha_4 + 5 \alpha_5 + 4 \alpha_6 + 3 \alpha_7 + 2 \alpha_8$\tabularnewline
$F_4$ & $\alpha_1$ & $ 2\alpha_1 + 3 \alpha_2 + 4\alpha_3 + 2 \alpha_4 $\tabularnewline
$G_2$ & $\alpha_1$ & $3 \alpha_1 + 2 \alpha_2 $ \tabularnewline
\hline
\end{tabular}
\vspace{0.2cm}
\caption{The simple root $\alpha_G$ and the highest root $\chi_G$ according to the Dynkin diagram of the restricted root system}
\label{table3}
\end{table}

\begin{proof}
The set \( \{ \alpha \in \Delta \mid ( \chi_G, \alpha) > 0\} \) is the set of simple roots connected to the added node in the extended Dynkin diagram.
The result is thus a consequence of the classification of those diagrams, see \eg \cite[Ch.\,VI, \S\,4, no.\,3]{Bourbaki}. 
Since \(\mathrm{Ad}\) is selfdual, \(\chi_{G}^{\star} = \chi_G\) and \(\langle \alpha, \chi_G\rangle\neq 0\) if and only if \(\langle \alpha^\star, \chi_G\rangle\neq 0\).
This proves \eqref{item:alphaG}. 

Let $\Gamma$ be a discrete (\resp word hyperbolic) group and $\rho : \Gamma\to G$ a representation.
Proposition~\ref{prop:theta-comp-Anosov} implies that \(\rho\) is \(P_{\{\alpha_G\}}\)-divergent (\resp \(P_{\{\alpha_G\}}\)-Anosov) if and only if \(\Ad \circ \rho: \Gamma \to \GL_\RR(\g)\) is \(P_d\)-divergent (\resp \(P_d\)-Anosov).
On the other hand, Lemma~\ref{lem:theta-comp-GL-Opq} implies that \(\Ad\circ \rho: \Gamma\to \OO(\kl)\) is \(P_{d}(\kl)\)-divergent (\resp \(P_{d}(\kl)\)-Anosov) if and only if \(\Ad \circ \rho: \Gamma \to \GL_\RR(\g)\) is \(P_d\)-divergent (\resp \(P_d\)-Anosov).
This proves \eqref{item:div-alphaG} and~\eqref{item:Ano-alphaG}.
\end{proof}

\section{Compactifying Riemannian locally symmetric spaces: the case of orthogonal groups}
\label{sec:comp-anos-locally}

In this section we construct a compactification for Riemannian locally symmetric spaces arising from \(P_1(b)\)-Anosov representations into $\OO(b)$.

For a nondegenerate symmetric bilinear form \(b\) of signature $(p,q)$ on a real vector space \(V\), the Riemannian symmetric space $X_{b}$ of $\OO(b)$ admits a realization as an open subset in the Grassmannian $\mathrm{Gr}_q(V)$, namely as the set of $W\in \mathrm{Gr}_q(V)$ such that the restriction of $b$ to $W\times W$ is negative definite.
Its closure 
\begin{equation} \label{eqn:overXb}
\overX_{b} = \{ W\in \mathrm{Gr}_q(V) \mid b(x,x)\leq 0\quad \forall x\in W\}
\end{equation}
in $\mathrm{Gr}_q(V)$ is a compactification of~$X_{b}$.
This compactification is isomorphic to a minimal Satake compactification of $X_{b}$ if $p>q$ (see Section~\ref{sec:compactification_Xpq}), and a generalized Satake compactification if $p=q$ (see Section~\ref{sec:generalizedSatake}).
The main result that we prove in this section is the following.

\begin{theorem}\label{thm:comp_loc_sym_spc_opq}
Let \(b\) be a nondegenerate symmetric bilinear form of signature \((p,q)\) on a real vector space~\(V\), where $p,q\in\NN^*$ and \((p,q)\neq (1,1), (2,2)\).
Let $\Gamma$ be a word hyperbolic group and \(\rho:\Gamma\to \OO(b)\) a \(P_1(b)\)-Anosov representation with boundary map \(\xi : \partial_\infty \Gamma\to\mathcal{F}_1(b)\).
  Let
  \[ \mathcal{N}_{\rho} = \bigcup_{\eta \in \partial_\infty \Gamma} \{ W\in\overX_{b} \mid \xi(\eta) \subset W\}. \]
  Then the action of \(\Gamma\) via~$\rho$ on $\Omega = \overX_{b} \smallsetminus \mathcal{N}_{\rho}$ is properly discontinuous and cocompact.
  The set \(\Omega\) contains the Riemannian symmetric space \(X_{b}\) and \(\rho(\Gamma)\backslash \Omega\) is a compactification of \(\rho(\Gamma)\backslash X_{b}\).
\end{theorem}

Properness will be proved in Section~\ref{sec:proof-prop-theor} and cocompactness in Section~\ref{sec:proof-cocomp-theor-1}.

\subsection{Nonpositive quadratic spaces}\label{sec:properties} 

We establish the following elementary lemma, used later in the proof of Theorem~\ref{thm:comp_loc_sym_spc_opq}.
We denote by
\[ \Ker(b) = \{ y \in V \mid b(x,y)=0\ \ \forall x \in V\} \]
the kernel of a symmetric bilinear form $b$ on a real vector space~$V$.

\begin{lemma} \label{lem:incid_iso_neg}
  Let \(b\) be a nondegenerate symmetric bilinear form of signature \((p,q)\) on a real vector space \(V\). Let \(W\in \mathrm{Gr}_q( V)\) satisfy $b(x,x)\leq 0$ for all $x\in W$.
  \begin{asparaenum}
  \item\label{item:nullW} If \(y\in W\) satisfies \(b(y,y)=0\), then \(y\in \Ker(b|_{W\times W})\).
    Conversely, if \(y\in V\) satisfies \(b(y,y)=0\) and \(b(x,y)=0\) for all \(x\in W\), then \(y\in W\).
  \item\label{item:incid} For any \(b\)-isotropic subspace $L$ of~$V$,
    \begin{center}
      \(L \cap W \neq \{0\}\) \(\Longleftrightarrow\) \(W + L^{\perp_{b}} \neq V\), and\\
      \(L \subset W \) \(\Longleftrightarrow\) \(W \subset L^{\perp_{b}} \).
          \end{center}
  \end{asparaenum}
\end{lemma}

\begin{proof}[Proof of Lemma~\ref{lem:incid_iso_neg}]
(1)\ If \(y\in W\) satisfies \(b(y,y)=0\), then $b(x,y)=0$ for all \(x\in W\), otherwise we would have $b(x+ty,x+ty) = b(x,x) + 2t b(x,y) > 0$ for certain values of $t\in\RR$.

Conversely, let $y\in V$ satisfy \(b(y,y)=0\) and \(b(x,y)=0\) for all \(x\in W\).
Let \(\ell = \RR y \subset V\).
The projection \(Z\) of \(W\) to \(\ell^{\perp_{b}}/\ell\) is a nonpositive subspace in a vector space equipped with a nondegenerate symmetric bilinear form of signature \((p-1, q-1)\).
Such a subspace \(Z\) has trivial intersection with any \((p-1)\)-dimensional positive subspace, hence $\dim_{\RR}(Z)\leq q-1$.
Since $\dim_{\RR}(W)=q$, we deduce~\(\ell\subset\nolinebreak W\) and \(y\in W\).

(2) Suppose $L\cap W\neq \{0\}$ and let \(y\) be non zero in \( L \cap W\).
Then $W\subset y^{\perp_{b}}$ by~\eqref{item:nullW}.
As \(L\) is isotropic, \(L^{\perp_{b}} \subset y^{\perp_{b}}\), and one gets \(W+ L^{\perp_{b}} \subset y^{\perp_{b}}\) and \(W + L^{\perp_{b}} \neq V\).

Conversely, suppose \(W+ L^{\perp_{b}}\neq V\). Let \(H\subset V\) be a hyperplane containing \(W\) and \(L^{\perp_{b}}\) and \( y\in V\) such that \(y^{\perp_{b}}=H\).
By duality \(y \in L\) and \(y\) is  isotropic.
By~\eqref{item:nullW} we have \(y\in W\), hence \(y \in W \cap L\) and \(W\cap L \neq \{0\}\). 
The equivalence \(L \subset W \) \(\Leftrightarrow\) \(W \subset L^{\perp_{b}} \) follows from~\eqref{item:nullW} as well.
\end{proof}

\subsection{Proper discontinuity}
\label{sec:proof-prop-theor}

Properness in Theorem~\ref{thm:comp_loc_sym_spc_opq} is an immediate consequence of Lemma~\ref{lem:anosov-lim-set} and of the following proposition with $i=1$.
We refer to Definitions \ref{def:lim-set} and~\ref{defi:theta_divergent-gp} for the notions of limit set and divergent representation, and to Section~\ref{sec:orthogonal-groups} and Example~\ref{ex:Ano-Opq} for the assumptions on~$i$.

\begin{proposition}\label{prop:propernessPk}
Let \(b\) be a nondegenerate symmetric bilinear form of signature \((p,q)\) on a real vector space~\(V\), where $p,q\in\NN^*$ and $(p,q)\neq (1,1)$.
Let  $1\leq i\leq\min(p,q)$; if $p=q$, assume that $i<p$.
Let $\Gamma$ be a discrete group and \(\rho:\Gamma\to \OO(b)\) a \(P_{i}(b)\)-divergent representation.
Let 
  \begin{equation*}
    \mathcal{W}^{i}_{\rho} = \bigcup_{L \in \mathcal{L}^{\mathcal{F}_i(b)}_{ \rho(\Gamma)}} \bigl\{ W\in \overX_{b} \mid L \cap W \neq \{0\}\bigr\},
  \end{equation*}
  where \(\mathcal{L}^{\mathcal{F}_i(b)}_{\rho(\Gamma)} \subset \mathcal{F}_i(b)\) is the limit set of \(\rho(\Gamma)\).
  Then \(\overX_{b} \smallsetminus \mathcal{W}^i_\rho\) contains $X_b$ and the action of \(\Gamma\) on \(\overX_{b} \smallsetminus \mathcal{W}^i_\rho\) is properly discontinuous.
\end{proposition}

In fact we prove the following very general statement.

\begin{proposition}\label{prop:propernessPk-general}
Let \(b\) be a nondegenerate symmetric bilinear form of signature \((p,q)\) on a real vector space \(V\), where $p,q\in\NN^{\ast}$ and $p\neq q$.
Let $\Gamma$ be a discrete group and \(\rho : \Gamma\to\OO(b)\) any representation with discrete image and finite kernel.
Let
  \begin{equation*}
    \mathcal{W}_\rho = \bigcup_{i\in I_{\rho}}\ \bigcup_{L \in \mathcal{L}^{\mathcal{F}_i(b)}_{ \rho(\Gamma)}} \bigl\{ W\in \overX_{b} \mid L \cap W \neq \{0\}\bigr\},
  \end{equation*}
  where $I_{\rho}\subset\{ 1,\dots,\min(p,q)\}$ is the set of integers $i$ such that $\langle\alpha_i,\mu(\rho(\Gamma))\rangle$ is unbounded, and \(\mathcal{L}^{\mathcal{F}_i(b)}_{\rho(\Gamma)} \subset \mathcal{F}_i(b)\) is the limit set of \(\rho(\Gamma)\).
  Then \(\overX_{b} \smallsetminus \mathcal{W}_\rho\) contains $X_b$ and the action of \(\Gamma\) on \(\overX_{b} \smallsetminus \mathcal{W}_\rho\) is properly discontinuous.
\end{proposition}

 \begin{remarks}\label{rem:properness}
 \begin{enumerate}
   \item \label{item:bordif-general}
     Proposition~\ref{prop:propernessPk-general} provides a bordification of 
     $\Lambda\backslash X_b$ locally modeled on \(\overX_b\) for \emph{any}
     discrete subgroup $\Lambda$ of $\OO(b)$.
   From this we deduce a bordification of $\Lambda\backslash G/K$
   as a manifold with corners for \emph{any} discrete subgroup $\Lambda$ of any
   semisimple Lie group~$G$ (Theorem~\ref{thm:bdf}).
   \item\label{item:1-rem-prop} In \cite{KapovichLeeb15}, bordifications are constructed, by a different method, for discrete subgroups $\Gamma$ of a simple group~$G$ that are \emph{uniformly $\tau_{mod}$-regular} for some facet $\tau_{mod}$ of~$\overline{\aaa}^+$.
   If we write $\tau_{mod} = \overline{\aaa}^+ \cap \bigcap_{\alpha\in\Delta\smallsetminus\theta} \Ker(\alpha)$ for some nonempty $\theta\subset\Delta$, then these are the discrete subgroups $\Gamma$ of~$G$ for which there exist $c,C>0$ such that for any $\alpha\in\theta$ and any $\gamma\in\Gamma$,
   \[ \langle\alpha,\mu(\gamma)\rangle \geq c\,\Vert\mu(\gamma)\Vert - C, \]
   where $\Vert\cdot\Vert$ is a fixed norm on~$\aaa$.
   In other words, $\Gamma$ is $P_{\theta}$-divergent with a linear rate of divergence.
\end{enumerate}
 \end{remarks}

Recall that two points \(x\) and \(x'\) of \(X\) are said to be \emph{dynamically related}
  if there exist a sequence \((x_n)_{n\in \NN}\) in \(X^\NN\) converging to \(x\)
  and a sequence \((\gamma_n)_{n\in\NN} \in \Gamma^\NN\) going to infinity (\ie leaving every finite subset of~$\Gamma$) such that the sequence \((\gamma_n\cdot x_n)_{n\in\NN}\)
  converges to~\(x'\).
  Propositions \ref{prop:propernessPk} and~\ref{prop:propernessPk-general} are immediate consequences of the following classical dynamical criterion for properness (see \eg \cite{Frances_Lorentzian} for a proof) and of the following lemma.
\begin{criterion} \label{crit:dyn-proper}
  A group \(\Gamma\)
  acts properly discontinuously on a Hausdorff topological space \(X\)
  if and only if no pairs of points of \(X\) are dynamically related.
\end{criterion}

\begin{lemma}\label{lem:propernessPk-seq}
In the setting of Proposition~\ref{prop:propernessPk}, consider an arbitrary sequence \((W_n)_{n\in\NN}\in (\overX_{b} \smallsetminus \mathcal{W}^i_\rho)^\NN\) converging to some \(W\in \overX_{b} \smallsetminus \mathcal{W}^i_\rho\) and an arbitrary $P_{i}(b)$-divergent sequence \((\rho(\gamma_n))_{n\in \NN}\in \Gamma^\NN\) such that \((W'_n=\rho(\gamma_n)\cdot W_n)_{n\in \NN}\) converges to some \(W'\in\mathrm{Gr}_q(V)\).
Then \(W'\in\mathcal{W}^i_\rho\).
\end{lemma}

\begin{proof}[Proof of Lemma~\ref{lem:propernessPk-seq}]
We write $\rho(\gamma_n) = k_n a_n \ell_n \in K \exp(\overline{\aaa}^+) K$.
Up to extracting, we can assume that the sequences \((k_n)_{n\in \NN}\) and \((\ell_n)_{n\in \NN}\) converge to some \(k_\infty,\ell_\infty\in K\), respectively.
By Definition~\ref{def:lim-set} of the limit set (see Section~\ref{sec:orthogonal-groups} and Remark~\ref{rem:div}.\eqref{item:rem-div-Xi-g-inv}),
\begin{equation*}
  L^+ := k_\infty \cdot (\RR e_1\oplus \cdots \oplus \RR e_i) \quad \text{and} \quad
  L^- := \ell_{\infty}^{-1} \cdot (\RR e_{p+q-i+1}\oplus \cdots \oplus\RR e_{p+q})
\end{equation*}
belong to the limit set \(\mathcal{L}^{\mathcal{F}_i(b)}_{\rho(\Gamma)}\).
The assumption \(W\cap L^- = \{0\}\) and Lemma~\ref{lem:incid_iso_neg} imply that \(W\) is not contained in \((L^-)^{\perp_{b}} = \ell_{\infty}^{-1}\cdot (\RR e_{i+1}\oplus \cdots \oplus\RR e_{p+q})\).
This means that there exist \(w_\infty\in W\) and \(c_1, \dots, c_{p+q}\in\RR\) such that
\begin{equation*}
  \ell_\infty\cdot w_\infty =  \sum_{j=1}^{p+q} c_j e_j,
\end{equation*}
and \((c_1, \dots, c_i) \neq 0\).
There is a sequence \((w_n)_{n\in \NN}\in V^\NN\) converging to \(w_\infty\) such that \(w_n \in W_n\) for all~\(n\).
The sequence \((\ell_n \cdot w_n)_{n\in\NN}\) converges to \(\ell_\infty\cdot w_\infty\).
We write \(\ell_n \cdot w_n = \sum_{j=1}^{p+q} c_{j,n} e_j\), thus \(\lim_n c_{j,n} = c_j\) for any \(j\in \{ 1, \dots, p+q\}\).

For \(n\in \NN\), let \(r_n\) be the inverse of the Euclidean norm of the vector \((e^{ \langle \varepsilon_j, \log a_n\rangle} \, c_{j,n})_{j=1, \dots, i} \in \RR^i\).
Set \(d_{j,n} := r_n\, e^{ \langle \varepsilon_j, \log a_n\rangle} \, c_{j,n}\), for \(j\in \{1, \dots, i\}\) and \(n\in \NN\).
Up to extracting a subsequence, the sequence of \(i\)-tuples \(  (d_{1,n}, \dots, d_{i,n} )_{n\in\NN}\) converges to some \((d_1, \dots, d_i)\) of norm \(1\) in~\(\RR^i\).
Consider \(j_0 \in \{ 1, \dots, i\}\) such that \(c_{j_0}\neq 0\), the sequence
\[ \bigl(r_n\, e^{ \langle \varepsilon_{j_0}, \log a_n\rangle}\bigr)_{n\in\NN} = \Bigl(\frac{d_{j_0,n}}{c_{j_0,n}}\Bigr)_{n\in\NN} \]
converges to \(d_{j_0}/c_{j_0}\) and is thus bounded.
This implies that for every \(j>i\) the sequence \((r_n\, e^{ \langle \varepsilon_j, \log a_n\rangle} \, c_{j,n})_{n\in\NN}\) converges to zero since 
\begin{align*}
r_n\, e^{ \langle \varepsilon_j, \log a_n\rangle} \, c_{j,n} & =  r_n\, e^{ \langle \varepsilon_{j_0}, \log a_n\rangle} \, c_{j,n} \, e^{ -\langle \varepsilon_{j_0} - \varepsilon_j, \log a_n\rangle}\\
& \leq  r_n\, e^{ \langle \varepsilon_{j_0}, \log a_n\rangle} \, c_{j,n} \, e^{ -\langle \alpha_i, \log a_n\rangle},
\end{align*}
which converges to~$0$ by \(P_{i}\)-divergence of $(\rho(\gamma_n))_{n\in\NN}$.

We claim that the sequence \((v_n)_{n\in \NN}\) defined by \[v_n = r_n\, \rho(\gamma_n)\cdot w_n \in W'_n,\text{ for all } n\in \NN,\] 
converges to \(v_\infty = k_\infty \cdot (d_1 e_1 + \cdots + d_i e_i)\in L^+\). 
Indeed,
\begin{align*}
  k_{n}^{-1} \cdot v_n 
  & = r_n a_n \ell_n \cdot w_n =  \sum_{j=1}^{p+q}  r_n c_{j,n}\,
    a_n \cdot e_j \\ 
  & =  \sum_{j=1}^{p+q} r_n \, c_{j,n} \, e^{\langle \varepsilon_j,
    \log a_n \rangle}\,  e_j  \underset{n\to \infty}{\longrightarrow} \sum_{j=1}^{i} d_j e_j.
\end{align*}
By the convergence of \((  W'_n)_{n\in \NN}\) to \(W'\), \(v_\infty\) belongs to \(W'\) as well. 
Hence \(W'\) has  nontrivial intersection with \(L^+\) and  belongs to \( \mathcal{W}^i_\rho\).
\end{proof}

\begin{proof}[Proof of Proposition~\ref{prop:propernessPk-general}]
Apply Criterion~\ref{crit:dyn-proper} and Lemma~\ref{lem:propernessPk-seq}, using the fact that if $\rho : \Gamma\to G$ is discrete with finite kernel, then for any sequence \((\gamma_n)_{n\in\NN} \in \Gamma^\NN\) going to infinity, up to passing to a subsequence, there exists $i$ such that $(\rho(\gamma_n))_{n\in\NN}$ is $P_{\{\alpha_i\}}$-divergent (by properness of the map~$\mu$). 
\end{proof}

\subsection{Compactness}
\label{sec:proof-cocomp-theor-1}

We now restrict to a special class of divergent representations, namely Anosov representations (Definition~\ref{defi:ano_theta}).
Compactness of $\rho(\Gamma)\backslash\Omega$ in Theorem~\ref{thm:comp_loc_sym_spc_opq} is a consequence of the following general result.

\begin{lemma} \label{lem:cocompactness}
    Let \(b\) be a nondegenerate symmetric bilinear form of signature \((p,q)\) on a real vector space~\(V\), where $p,q\in\NN^*$ and $(p,q)\neq (1,1)$.
  Let $1\leq i\leq\min(p,q)$; if $p=q$, assume that $i<p-1$.
For any word hyperbolic group $\Gamma$ and any \(P_i(b)\)-Anosov representation \(\rho :\nolinebreak\Gamma\to\nolinebreak\OO(b)\) with boundary map \(\xi_i : \partial_\infty \Gamma \to \mathcal{F}_i( b)\), let
  \begin{equation*}
    \mathcal{V}_\rho = \bigcup_{\eta \in \partial_\infty \Gamma} \bigl\{ W\in \overX_{b} \mid \xi_i(\eta ) \subset  W \bigr\}.
  \end{equation*}
  Then the action of \(\Gamma\) on \(\overX_{b} \smallsetminus \mathcal{V}_\rho\) is cocompact.
\end{lemma}

Lemma~\ref{lem:cocompactness} itself is a direct consequence of Lemma~\ref{lem:theta-comp-GL-Opq}.\eqref{item:theta-comp-GL-Opq-Ano}, of the fact that \(\overX_b \subset \mathrm{Gr}_q(V)\) is closed, and of the following statement.

\begin{lemma} \label{lem:compactness-GLV}
  Let \(V\) be a real vector space of dimension~$n$, let \(1\leq i \leq q \leq n-1\) be two integers, and let \(\rho: \Gamma \to \GL_{\RR}(V)\) be a \(P_i\)-Anosov representation with boundary maps \(\xi_i : \partial_\infty \Gamma \to \mathrm{Gr}_i(V)\) and \(\xi_{n-i}: \partial_\infty \Gamma \to \mathrm{Gr}_{n-i}(V)\).
Let
\begin{equation*}
  \mathcal{B}_\rho = \bigcup_{\eta \in \partial_\infty \Gamma} \bigl\{ W\in \mathrm{Gr}_q(V) \mid \xi_i(\eta)\subset W \bigr\}.
\end{equation*}
Then the action of \(\Gamma\) on \(\mathrm{Gr}_q( V) \smallsetminus \mathcal{B}_\rho\) is cocompact.
\end{lemma}

The rest of this section is devoted to proving Lemma~\ref{lem:compactness-GLV}.
Let us first introduce some notation.
In a metric space \((X, d_X)\) the distance from a point to a set will be denoted by \(\dist_X\) and the Hausdorff distances by \(\haus_X\).
For lines \(L, L'\in\PP(V)\) with respective direction vectors \(v\) and~\(v'\), we set
\begin{equation*}
  d_{\PP(V)}( L, L') := | \sin \measuredangle (v,v') |.
\end{equation*}
For \(W, W' \in \mathrm{Gr}_q(V)\), we set
\begin{equation*}
  d_{\mathrm{Gr}_q(V)}( W,W') = \haus_{\PP(V)} ( \PP(W), \PP(W')).
\end{equation*}
For \(L\in\PP(V)\), we set
\begin{equation*}
  \mathcal{K}_L := \bigl\{ W \in \mathrm{Gr}_q(V) \mid L \subset W\bigr\}.
\end{equation*}
Note that $\mathcal{K}_{g\cdot L} = g\cdot\mathcal{K}_L$ for all $g\in\GL_{\RR}(V)$.
The following identity is easily established:
\begin{equation}
  \label{eq:dist}
  \dist_{\mathrm{Gr}_q(V)}( W, \mathcal{K}_L) = \dist_{\PP(V)}( L, \PP(W)).
\end{equation}
The following result is a consequence of estimates established in \cite{GGKW_anosov}.

\begin{proposition} \label{prop:P1-expanding}
  Let \(V\) be a real vector space of dimension~$n$, let \(\rho: \Gamma \to \GL_{\RR}(V)\) be a \(P_1\)-Anosov representation with boundary map \(\xi_1: \partial_\infty \Gamma \to \PP(V)\), and let \(1\leq q \leq n-1\).
  Then for any \(\eta\in\partial_\infty \Gamma\) and any $c>1$, there exist \(\gamma\in \Gamma\) and an open subset \(U\) of \(\mathrm{Gr}_q(V)\) containing \(\mathcal{K}_{\xi_1(\eta)}\) such that
  \begin{equation*}
    \dist_{\mathrm{Gr}_q(V)} \bigl( \rho(\gamma)\cdot W,  \rho(\gamma)\cdot \mathcal{K}_L \bigr) \geq c\, \dist_{\mathrm{Gr}_q(V)}( W, \mathcal{K}_L )
  \end{equation*}
  for all \(W\in U\) and \(L\in\PP(V)\) with \(\mathcal{K}_L\subset U\).
\end{proposition}

\begin{proof}[Proof of Proposition~\ref{prop:P1-expanding}]
Let \((e_1, \dots, e_n)\) be a basis of \(V\) so that $G=\GL_{\RR}(V)$ admits a Cartan decomposition $G=K \exp(\overline{\aaa}^+) K$ as in Example~\ref{sec:gener-line-groups}.
Fix $\eta\in\partial_{\infty}\Gamma$ and let \((\gamma_n)_{n\in \NN}\) be a quasigeodesic ray in \(\Gamma\) converging to~\(\eta\).
For any $n\in\NN$ we write $\rho(\gamma_n) = k_n a_n \ell_n \in K \exp(\overline{\aaa}^+) K$. 
Let \(x_0 =\RR e_1 \in \PP(V)\).
By \cite[Th.\,1.3.(1)$\Rightarrow$(4) \& Th.\,5.1.(1)]{GGKW_anosov},
\begin{equation*}
 \xi_1( \eta) = \lim_{n\to \infty} k_n \cdot x_0.
\end{equation*}
By \cite[\S\,5.4.1]{GGKW_anosov} (see the proof of Prop.\,5.11 of \cite{GGKW_anosov}, establishing (5.8) in Lem.\,5.12), there exist \(\delta>0\) and $N\in\NN$ such that for all $n\geq N$,
\begin{equation*}
  \limsup_{m\to \infty} d_{\PP(V)} \bigl( x_0, (a_{n}^{-1} k_{n}^{-1} k_{n+m} a_n) \cdot x_0\bigr ) \leq 1-\delta.
\end{equation*}
Since \(a_n \cdot x_0 = x_0\) and since the metric $d_{\PP(V)}$ is invariant under $\ell_n^{-1}\in K$, we deduce that
\[ d_{\PP(V)}\big( \ell_{n}^{-1}\cdot x_0, \rho(\gamma_n)^{-1} \cdot \xi_1(\eta)\big) \leq 1- \delta \]
for all $n\geq N$.
We set
\[ \mathcal{U}_n = \left\{ W \in \mathrm{Gr}_q(V) \mid \dist_{\PP(V)} \big(\ell_{n}^{-1} \cdot x_0, \PP(W)\big) < 1 - \frac{\delta}{2}\right\}. \]
By~\eqref{eq:dist}, for any $L\in\PP(V)$ we have
\begin{equation} \label{eqn:KLsubsetUn}
\mathcal{K}_L \subset \mathcal{U}_n\, \Longleftrightarrow\, d_{\PP(V)}( \ell_{n}^{-1} \cdot x_0, L) < 1- \frac{\delta}{2}.
\end{equation}
Therefore, \(\mathcal{K}_{\rho(\gamma_n)^{-1}\cdot \xi_1(\eta)}\subset \mathcal{U}_n\)
for all $n\geq N$.
To conclude the proof, it is enough to establish the following.

\begin{claim} \label{cla:expansion}
For any \(c>1\), there exists \(n_c\in\NN\) such that for all \(n\geq n_c\), all \(W\in\mathcal{U}_n\), and all \(L\in \PP(V)\) with \(\mathcal{K}_L\subset\mathcal{U}_n\),
  \begin{equation} \label{item:WL}
    \dist_{\mathrm{Gr}_q(V)} ( W, \mathcal{K}_{ L}) \geq c \, \dist_{\mathrm{Gr}_q(V)} \bigl( \rho(\gamma_n) \cdot W, \rho(\gamma_n) \cdot \mathcal{K}_{L}\bigr).
  \end{equation}
\end{claim}

\noindent
Indeed, Proposition~\ref{prop:P1-expanding} follows from Claim~\ref{cla:expansion} by setting \(\gamma = \gamma_{n}^{-1}\) and \(U=\rho(\gamma_n)\cdot\mathcal{U}_n\) for $n\geq\max(N,n_c)$.

We now prove Claim~\ref{cla:expansion}.
It is a consequence of the following two elementary estimates:
\begin{enumerate}[(a)]
\item\label{item:a} There exist \(\epsilon, M >0\) such that for any \(L'\in \PP(V)\) and \(W'\in \mathrm{Gr}_q(V)\), if
\(d_{\PP(V)}(x_0, L') < 1- \delta/2\), then
  \begin{equation*}
    \dist_{\PP(V)}\bigl (L', \PP(W') \cap \overline{B}_{\PP(V)}(x_0, 1- \epsilon)\bigr) \leq M\, \dist_{\PP(V)}(L', \PP(W')).
  \end{equation*}
\item\label{item:b} For any \(\epsilon>0\) and \(c'>0\), there exists \(n_{\epsilon,c'}\) such that \(a_n|_{\overline{B}_{\PP(V)}(x_0, 1- \epsilon)}\) is \(c'\)-contracting for all \(n\geq n_{\epsilon,c'}\).
\end{enumerate}
Indeed, fix \(c>1\).
For $n\geq N$, consider \(W\in \mathcal{U}_n\) and \(L\in \PP(V)\) with \(\mathcal{K}_L\subset \mathcal{U}_n\).
Set \(W' = \ell_n \cdot W\) and \(L' = \ell_n \cdot L\).
By \eqref{eqn:KLsubsetUn}, we have \( d_{\PP(V)}(x_0, L') < 1 - \frac{\delta}{2}\).
By \eqref{item:a} and~\eqref{item:b}, if $n\geq n_{\epsilon,c/M}$, then
\begin{align*}
    \dist_{\PP(V)}  ( a_n \cdot L', &\, a_n \cdot \PP(W'))\\
   \leq &\, \dist_{\PP(V)}\big( a_n \cdot L', a_n \cdot \big(\PP(W') \cap \overline{B}_{\PP(V)}(x_0, 1- \epsilon)\big)\big)\\
   \leq &\, \frac{c}{M} \ \dist_{\PP(V)} \big( L', \big(\PP(W') \cap \overline{B}_{\PP(V)}(x_0, 1- \epsilon)\big)\big)\\
   \leq &\, c \ \dist_{\PP(V)} ( L', (\PP(W')).
\end{align*}
On the other hand, by \eqref{eq:dist} and the fact that the metric $d_{\mathrm{Gr}_q(V)}$ is invariant under~$K$,
\begin{equation*}
  \left\{\begin{array}{l}
 \dist_{\mathrm{Gr}_q(V)}( W, \mathcal{K}_L) = \dist_{\PP(V)}( L', \PP(W')),\\
  \dist_{\mathrm{Gr}_q(V)}( \rho(\gamma_n)\cdot W, \rho(\gamma_n)\cdot\mathcal{K}_{L}) = \dist_{\PP(V)}( a_n \cdot L', a_n \cdot \PP(W')).
  \end{array}\right.
\end{equation*}
This concludes the proof of Claim~\ref{cla:expansion}.

For the sake of completeness, we now give a proof of \eqref{item:a} and~\eqref{item:b} above: \eqref{item:b} is a consequence of the fact that \(\langle \varepsilon_1 - \varepsilon_2, \log a_n \rangle \to +\infty\) (since $\rho$ is $P_1$-divergent by Definition~\ref{defi:ano_theta} of a $P_1$-Anosov representation).
For \eqref{item:a} we argue by contradiction: suppose that there are sequences \((\epsilon_m)_{m\in \NN} \in (\RR_{>0})^\NN\), \((M_m)_{m\in \NN} \in (\RR_{>0})^\NN\), \((L'_m)_{m\in \NN} \in \PP(V)^\NN\), and \((W'_m)_{m\in \NN} \in \mathrm{Gr}_q(V)^\NN\) such that \((\epsilon_m)_{m\in \NN}\) converges to \(0\), such that \((M_m)_{m\in\NN}\) diverges to \(+\infty\), and such that for any \(m\) we have \(d_{\PP(V)}( x_0, L'_m) < 1- \delta/2\) and
\begin{equation} \label{eqn:contradict-Mn}
\dist_{\PP(V)} \bigl( L'_m, \PP(W'_m) \cap \overline{B}_{\PP(V)}( x_0, 1- \epsilon_m)\bigr) > M_m \, \dist_{\PP(V)}( L'_m, \PP(W'_m)).
\end{equation}
For any~\(m\), the left-hand side of \eqref{eqn:contradict-Mn} is $\leq 1$, hence \(\dist_{\PP(V)}( L'_m, \PP(W'_m)) \leq 1/M_m\).
Let \(D_m\subset W'_m\) be a line such that \(d_{\PP(V)}(L'_m, D_m) = \dist_{\PP(V)}(L'_m, \PP(W'_m))\).
Then \( d_{\PP(V)}(x_0,D_m) \leq d_{\PP(V)}(x_0,L'_m) + d_{\PP(V)}(L'_m, D_m) \leq 1 - \delta/2 + 1/M_m\), and so \(D_m\) belongs to \(\overline{B}_{\PP(V)}( x_0, 1- \epsilon_m)\) for large enough~$m$, contradicting \eqref{eqn:contradict-Mn}.
\end{proof}

Lemma~\ref{lem:compactness-GLV} is a consequence of Proposition~\ref{prop:P1-expanding} and of the following dynamical compactness criterion from \cite{KapovichLeebPorti}, inspired by Sullivan's dynamical characterization of convex cocompactness \cite{Sullivan_85}.
We recall the proof for the reader's convenience.

\begin{lemma}[{\cite[Prop.\,2.5]{KapovichLeebPorti}}]
  \label{lem:compactness-metric}
  Let \(\Lambda\) be a group acting by homeomorphisms on a compact metric space \((Z, d_Z)\) and on a compact set \(D\).
  Let \(E\) be a closed \(\Lambda\)-invariant subset of~$Z$ fibering equivariantly over~\(D\),
  with fibers denoted by \(E_d\), \(d\in D\).
  Suppose that for any \(d\in D\) there exist an element \(\gamma\in \Lambda\), an open set \(U\subset Z\) containing \(E_d\), and a constant \(c>1\) such that
  \begin{equation} \label{eqn:distZ-E}
  \dist_Z ( \gamma \cdot z, \gamma \cdot E_{d'}) \geq c\, \dist_Z( z, E_{d'})
  \end{equation}
  for all \(z\in U\) and \(d'\in D\) with \(E_{d'}\subset U\).
Then the action of \(\Lambda\) on \(Z \smallsetminus E\) is cocompact.
\end{lemma}

\begin{proof}[Proof of Lemma~\ref{lem:compactness-metric}]
Suppose by contradiction that the action is not cocompact, and let \((\epsilon_n)_{n\in \NN}\) be a sequence converging to~\(0\).
For any \(n\in\NN^*\), the set \(C_n = \{ z\in Z \mid \dist_Z( z, E) \geq \epsilon_n\}\) is compact, hence there exists a \(\Lambda\)-orbit contained in \(Z \smallsetminus (C_n \cup E)\); by approaching the supremum of $\dist_Z(\cdot,E)$ on this orbit, we find an element \(z_n\in Z\) such that $0 < \dist_Z(z_n , E) \leq \epsilon_n$ and
\[ \dist_Z( \gamma\cdot z_n, E) \leq (1+\epsilon_n) \, \dist_Z(z_n, E) \quad \forall \gamma \in \Lambda.\]
Up to extracting, we may assume that \((z_n)_{n\in\NN^*}\) converges to some \(z_{\infty}\in E\), belonging to a fiber $E_d$, $d\in D$.
Let \((\gamma,U,c)\) be such that \eqref{eqn:distZ-E} holds for all \(z\in U\) and \(d'\in D\) with \(E_{d'}\subset U\).
For any $n\in\NN^*$, consider $d_n\in D$ such that \(\dist_Z(\gamma\cdot z_n, \gamma\cdot E_{d_n})\) is minimal, equal to \(\dist_Z(\gamma\cdot z_n, E)\); the sequence $(d_n)_{n\in\NN^*}$ converges to~$d$ since $(z_n)_{n\in\NN^*}$ converges to $z_{\infty}$.
For large enough~\(n\) we have \(E_{d_n} \subset U\), and so
\begin{align*}
  \dist_Z( \gamma \cdot z_n, E) & = \dist_Z( \gamma \cdot z_n, \gamma \cdot E_{d_n}) \\
& \geq c\, \dist_Z( z_n, E_{d_n}) \geq c\, \dist_Z( z_n, E)\\
  & \geq \frac{c}{1+\epsilon_n}\, \dist_Z( \gamma\cdot z_n, E).
\end{align*}
This is impossible since \(c/(1+\epsilon_n) > 1\) for large enough~\(n\).
\end{proof}

\begin{proof}[Proof of Lemma~\ref{lem:compactness-GLV}]
If $i=1$, then Lemma~\ref{lem:compactness-GLV} is an immediate consequence of Proposition~\ref{prop:P1-expanding} and Lemma~\ref{lem:compactness-metric} with \((Z,E,D) = (\mathrm{Gr}_q(V),\mathcal{B}_{\rho},\partial_{\infty}\Gamma)\).
Suppose now that $i$ is arbitrary and let \(\tau_i : \GL_{\RR}(V) \to \GL_{\RR}( \bigwedge^i V)\) be the homomorphism coming from the action of \(\GL_{\RR}(V)\) on the \(i\)-th exterior power of~\(V\).
By Proposition~\ref{prop:theta-comp-Anosov}, the representation \(\rho\) is \(P_i\)-Anosov if and only if \(\tau_i\circ\rho: \Gamma \to \GL_\RR( \bigwedge^i V)\) is \(P_1\)-Anosov.
The map
\begin{align*}
\mathrm{Gr}_q(V) & \longrightarrow
                   \mathrm{Gr}_{\binom{q}{i}}\left({\bigwedge}^i V\right) \\ W & \longmapsto \bigwedge^i W
\end{align*}
is \(\tau_i\)-equivariant and injective.
Moreover, for \(E\in \mathrm{Gr}_i(V)\) and \(W\in \mathrm{Gr}_q(V)\) we have \(E\subset W\) if and only if \(\bigwedge^i E\subset \bigwedge^i W\).
Therefore Lemma~\ref{lem:compactness-GLV} for $\rho$ follows from Lemma~\ref{lem:compactness-GLV} for $\tau_i\circ\rho$ with $i=1$.
\end{proof}

\begin{proof}[Proof of Theorem~\ref{thm:comp_loc_sym_spc_opq}]
  By Proposition~\ref{prop:propernessPk}, the action of \(\Gamma\) on \(\Omega= \overX_{b} \smallsetminus \mathcal{N}_\rho\) is properly discontinuous since \(\mathcal{N}_\rho = \mathcal{W}_\rho\) for a \(P_1(b)\)-Anosov representation. 
  As \( \mathcal{N}_\rho = \mathcal{V}_\rho\), cocompactness of the action follows from Lemma~\ref{lem:cocompactness}.
\end{proof}

\section{Compactifying Riemannian locally symmetric spaces:\\ the general case}
\label{sec:comp-riem-locally-gal}

We now use the compactification of Riemannian locally symmetric spaces of indefinite orthogonal groups constructed in Theorem~\ref{thm:comp_loc_sym_spc_opq}, Proposition~\ref{prop:propernessPk}, and Lemma~\ref{lem:cocompactness} to prove Theorem~\ref{thm:main} in Section~\ref{subsec:proof-thm-main}, Theorem~\ref{thm:tame} in Section~\ref{subsec:tameness}, and a more precise version of Theorem~\ref{thm:main_max_Satake} in Section~\ref{sec:maxSatake}.
The case of complex orthogonal groups is investigated in Section~\ref{sec:compl-orth-groups}.

\subsection{The subalgebra compactification}
\label{sec:subalg-comp}

Let \(G\) be a real semisimple Lie group.
It acts on its Lie algebra~$\g$ via the adjoint action, preserving the Killing form~$\kl$ (see Section~\ref{sec:adjoint}).
As in Section~\ref{sec:comp-anos-locally}, the Riemannian symmetric space $X_{\kl}$ of the orthogonal group $\OO(\kl)$ admits a realization as an open subset in the Grassmannian $\mathrm{Gr}_{\dim\mathfrak{k}}(V)$, namely as the set of $W\in \mathrm{Gr}_{\dim\mathfrak{k}}(V)$ such that the restriction of $\kl$ to $W\times W$ is negative definite.
Its closure
\begin{equation*}
\overX_{\kl} = \{ W\in \mathrm{Gr}_{\dim\mathfrak{k}}(\g) \mid \kl(x,x)\leq 0\quad \forall x\in W\}
\end{equation*}
is a compactification of~$X_{\kl}$.

The element \(\mathfrak{r}_\emptyset := \mathfrak{k}\) belongs to \(X_{\kl} \subset \mathrm{Gr}_{\dim \mathfrak{k}}( \mathfrak{g})\) and its stabilizer in \(G\) is~\(K\).
Thus the orbit \(\Ad(G)\cdot \mathfrak{r}_\emptyset\) in \(X_{\kl}\) identifies with the Riemannian symmetric space \(X= G/K\).
The closure \(\overX^{sba}\) of \(X \simeq \Ad(G)\cdot \mathfrak{r}_\emptyset\) in \(\overX_{\kl}\) is called the \emph{subalgebra compactification} of~\(X\).

\begin{proposition}[{\cite[Th.\,1.1]{Ji_Lu}}]\label{prop:Satake_subalgebra}
The subalgebra compactification of~$X$ is isomorphic to the maximal Satake compactification of~$X$. 
\end{proposition}

We now describe representatives of the finitely many orbits \(G\)-orbits in~\(\overX^{sba}\).
For \(\theta\subset \Delta\) the Lie algebra \(\mathfrak{p}_\theta\) of \(P_\theta\) has nilpotent radical \(\mathfrak{u}_\theta\) (see Section~\ref{sec:parabolic-subgroups}) and a Levi component is
\[\mathfrak{l}_\theta = \mathfrak{g}_0 \oplus \!\!\!\!\!\!\! \bigoplus_{\alpha \in \Sigma \cap \mathrm{span}( \Delta \smallsetminus \theta)} \mathfrak{g}_\alpha,\]
and \( \mathfrak{k}_\theta := \mathfrak{k} \cap \mathfrak{l}_\theta\) is a maximal compact subalgebra of \(\mathfrak{l}_\theta\).
Set 
\(\mathfrak{r}_\theta = \mathfrak{k}_\theta \oplus \mathfrak{u}_\theta.\)
(For \(\theta = \emptyset\), one has indeed \(\mathfrak{r}_\theta = \mathfrak{k}\).)

\begin{lemma}[\cite{Ji_Lu}]
  \label{lem:orbit-Xsa}
  The subalgebra compactification \(\overX^{sba}\) is the disjoint union \(\bigcup_{ \theta \subset \Delta} \Ad(G) \cdot \mathfrak{r}_\theta\).
\end{lemma}

The kernel of the restriction of \(\kl\) to \(\mathfrak{r}_\theta\) is precisely \(\mathfrak{u}_\theta\):

\begin{lemma}
  \label{lem:ker_Xsa}
  For every \(\theta \subset \Delta\), one has \(\Ker(\kl|_{\mathfrak{r}_\theta\times \mathfrak{r}_\theta}) = \mathfrak{u}_\theta\). Consequently a nilpotent element \(Y\) belongs to \(\mathfrak{r}_\theta\) if and only if it belongs to \(\mathfrak{u}_\theta\).
\end{lemma}

\begin{proof}
  Since \(\mathfrak{r}_\theta \subset \mathfrak{p}_\theta\), one has \(\mathfrak{u}_\theta = \Ker(\kl|_{\mathfrak{p}_\theta\times\mathfrak{p}_\theta}) \subset \Ker(\kl|_{\mathfrak{r}_\theta\times \mathfrak{r}_\theta})\).
  Furthermore, since the Killing form $\kl$ is negative definite in restriction to~\(\mathfrak{k}\), the intersection \(\Ker(\kl|_{\mathfrak{r}_\theta\times \mathfrak{r}_\theta}) \cap \mathfrak{k}_\theta\) is trivial, hence \(\Ker(\kl|_{\mathfrak{r}_\theta\times \mathfrak{r}_\theta}) = \mathfrak{u}_\theta\).

  A nilpotent element \(Y\) satisfies \(\kl(Y,Y)=0\), hence \(Y\in \Ker(\kl|_{\mathfrak{r}_\theta\times \mathfrak{r}_\theta})\) by Lemma~\ref{lem:incid_iso_neg}.
\end{proof}

\subsection{Proof of Theorem \ref{thm:main}} \label{subsec:proof-thm-main}

Let $\rho : \Gamma\to G$ be a $P_{\theta}$-Anosov representation.
By Proposition~\ref{prop:PthetaG_P1opq}, there exists a homomorphism \(\tau: G\to \OO(b)\) such that \(\tau\circ \rho:\Gamma \to \OO(b)\) is \(P_1(b)\)-Anosov.
Let $\Omega  = \overX_{b} \smallsetminus \mathcal{N}_{\rho}$ be the set given by Theorem~\ref{thm:comp_loc_sym_spc_opq}, on which $\Gamma$ acts properly discontinuously and cocompactly via $\tau\circ\rho$.
Let $x_0$ be a point of~$X_{b}$ whose stabilizer in $G$ is~$K$, and let $Y$ be the $\tau(G)$-orbit of~$x_0$: it identifies with $\tau(G)/\tau(K)$.
The closure $\overY$ of $Y$ in~$\overX_{b}$ is a generalized Satake compactification, see Lemma~\ref{lem:genSatake-subsymm}.
The group $\Gamma$ acts properly discontinuously and cocompactly via $\tau\circ\rho$ on $\Omega\cap\overY$.
The quotient $M  = (\tau\circ\rho)(\Gamma) \backslash (\Omega \cap \overY)$ thus gives a compactification of $(\tau\circ\rho)(\Gamma) \backslash \tau(G)/\tau(K)$.

By compactness of $\overX^{sba}$, the action of \(\Gamma\) on \(\Omega \times \overX^{sba}\) via $(\tau \times \Ad) \circ\rho$ is properly discontinuous and cocompact.
The $\Ad(G)$-orbit of $\mathfrak{r}_\emptyset$ in $\overX^{sba}$ identifies with $X=G/K$ and is dense.
Let $Z$ be the $(\tau\times\Ad)(G)$-orbit of $(x_0,\mathfrak{r}_\emptyset)$.
By Lemma~\ref{lem:gener-satake-comp-product}, the closure \(\overZ\) of $Z$ in $\overX_{b}\times\overX^{sba}$ is a generalized Satake compactification of~$Z$.
By construction, it dominates the maximal Satake compactification $\overX^{sba}$ of~$X$.
Lemma~\ref{lem:gal-sata-corners} says that it is a manifold with corners.
The group $\Gamma$ acts properly discontinuously and cocompactly on $(\Omega\times \overX^{sba})\cap\overZ$ via $(\tau\times\Ad)\circ\rho$.
The quotient 
\[ ((\tau\times\Ad)\circ\rho)(\Gamma) \backslash ((\Omega\times\overX^{sba})\cap\overZ) \]
gives a compactification of $((\tau\times\Ad)\circ\rho)(\Gamma) \backslash  Z \simeq \rho(\Gamma) \backslash X$ with the required properties.
This completes the proof of Theorem~\ref{thm:main}.

\subsection{Domains of discontinuity in $\overX^{sba}$ for discrete subgroups}
\label{sec:doma-disc-overxsba}

\begin{theorem} \label{thm:bdf}
Let $X=G/K$ be a Riemannian symmetric space where $G$ is a noncompact real semisimple Lie group.
For any discrete subgroup $\Lambda$ of~$G$ there is a closed $\Lambda$-invariant subset $\mathcal{N}$ of $\overX^{sba}\smallsetminus X$ such that $\Lambda$ acts properly discontinuously on $\overX^{sba}\smallsetminus\mathcal{N}$.
In particular 
\(\Lambda \backslash (\overX^{sba}\smallsetminus\mathcal{N})\) is a manifold with corners containing \(\Lambda \backslash X\) as a dense subset.
\end{theorem}

\begin{proof}[Proof of Theorem~\ref{thm:bdf}]
Let $\Ad : G\to\OO(\kl)$ be the adjoint action; it has finite kernel.
Let $\Lambda$ be a discrete subgroup of~$G$ and $\Omega = \overX_{\kl} \smallsetminus \mathcal{W}_{\Ad} \subset \overX_{\kl}$ the set given by Proposition~\ref{prop:propernessPk-general} for $\Gamma=\Lambda$ and $\rho=\Ad$, on which $\Lambda$ acts properly discontinuously via $\Ad$.
The $\Ad(G)$-orbit of~$\mathfrak{r}_\emptyset$ identifies with $X= G / K$.
The closure $\overX^{sba}$ of $X$ in~$\overX_{\kl}$ is the subalgebra compactification.
The group $\Lambda$ acts properly discontinuously via $\Ad$ on $\Omega\cap\overX^{sba}$.
The quotient $M = \Ad(\Lambda) \backslash (\Omega \cap \overX^{sba})$ thus is locally modeled on \(\overX^{sba}\) and contains $\Ad(\Lambda) \backslash X$.
\end{proof}

\subsection{Topological tameness} \label{subsec:tameness}

Theorem~\ref{thm:tame} is now a direct consequence of the construction of our compactification in Theorem~\ref{thm:main} and of the following proposition, which is proved in \cite{GGKW_compact}.

\begin{proposition}[{\cite[Prop.\,6.1]{GGKW_compact}}] \label{prop:semialgebraic}
  Let \(X\) be a real semi-algebraic set and \(\Gamma\) a torsion-free discrete group acting on~$X$ by real algebraic homeomorphisms.
  Suppose \(\Gamma\) acts properly discontinuously and cocompactly on some open subset \(\Omega\) of~\(X\).
  Let \(\mathcal{U}\) be a \(\Gamma\)-invariant real semi-algebraic subset of~$X$ contained in \(\Omega\) (\eg an orbit of a real algebraic group containing \(\Gamma\) and acting algebraically on~\(X\)).
Then the closure \(\overline{\mathcal{U}}\) of $\mathcal{U}$ in~$X$ is real semi-algebraic and \(\Gamma
  \backslash (\overline{\mathcal{U}} \cap \Omega)\) is compact and has a triangulation such that \(\Gamma \backslash (\partial \mathcal{U} \cap \Omega)\) is a finite union of simplices.
If \(\mathcal{U}\) is a manifold, then \(\Gamma \backslash \mathcal{U}\) is topologically tame.
\end{proposition}

\subsection{Compactifications modeled on the maximal Satake compactification} \label{sec:maxSatake}

We now prove the following theorem, which, together with Proposition~\ref{prop:Satake_subalgebra}, implies Theorem~\ref{thm:main_max_Satake} in the case that $G$ is simple.

\begin{theorem}\label{thm:max_Satake}
Let $G$ be a real simple Lie group.
Let $\alpha_G\in \Delta$ be the simple restricted root given by Proposition~\ref{prop:G_makeP1} (see Table~\ref{table3}) and \(d =\dim_\RR \mathfrak{g}_{\chi_G}\), where $\chi_G\in\Sigma^+$ is the highest restricted root.
Let $\Gamma$ be a word hyperbolic group and $\rho: \Gamma \to G$ a $P_{\{\alpha_G\}}$-Anosov representation.
Let \(\xi_d : \partial_\infty \Gamma \to \mathcal{F}_d( \kl)\) be the boundary map of \(\Ad \circ \rho\) (see Proposition~\ref{prop:G_makeP1}) and
\begin{equation*}
  \mathcal{N}_\rho = \bigcup_{\eta \in \partial_\infty \Gamma}\{ W \in \overX^{sba} \mid \xi_d(\eta) \subset W\}.
\end{equation*}
Then \(\Omega := \overX^{sba} \smallsetminus \mathcal{N}_\rho\) contains \(X = G/K\) and the action of \(\Gamma\) on \(\Omega\) is properly discontinuous and cocompact.
\end{theorem}

\begin{proof}
By Proposition~\ref{prop:propernessPk}, the action of \(\Gamma\) via  \(\Ad\circ \rho\) on \(\overX_{\kl} \smallsetminus \mathcal{W}_{\Ad \circ \rho}\) is properly discontinuous. 
Thus the action of \(\Gamma\) on \(\overX^{sba} \smallsetminus ( \mathcal{W}_{\Ad \circ \rho} \cap \overX^{sba})\) is properly discontinuous. 

Let us prove that \(\mathcal{N}_\rho = \mathcal{W}_{\Ad \circ \rho} \cap \overX^{sba}\) and that \(\mathcal{N}_\rho\) does not intersect \(X\).
The inclusion \(\mathcal{N}_\rho \subset \mathcal{W}_{\Ad \circ \rho} \cap \overX^{sba}\) is obvious. Let now \(W \in \mathcal{W}_{\Ad \circ \rho} \cap \overX^{sba}\), \ie there is \(\eta \in \partial_\infty \Gamma\) such that \(W \cap \xi_d(\eta) \neq \{0\}\).
By Proposition~\ref{prop:G_makeP1} \(\xi_d(\eta) \in \Ad(G)\cdot \g_{\chi_G}\).
Lemma~\ref{lem:inter-impl-contained} implies that \(W\notin X\), hence \(\mathcal{N}_\rho \cap X=\emptyset\), and that \(\xi_d( \eta) \subset W\), \ie \(W\in \mathcal{N}_\rho\).
Since \(\mathcal{V}_{\Ad\circ \rho} \cap \overX^{sba} = \mathcal{N}_\rho\), Lemma~\ref{lem:cocompactness} implies that the action of \(\Gamma\) on \(\Omega\) is cocompact.
\end{proof}

\begin{lemma}
  \label{lem:inter-impl-contained}
  Let \(W\in\overX^{sba}\) and \(L\in \Ad(G)\cdot \g_{\chi_G} \subset \mathrm{Gr}_d( \g)\).
  If \(L\cap W \neq \{0\}\), then \(W\notin X\) and \(L \subset \Ker(\kl|_{W\times W})\subset W\).
\end{lemma}

\begin{proof}[Proof of Lemma~\ref{lem:inter-impl-contained}]
  Since the elements of~\(L\) are nilpotent, the hypothesis implies that \(\Ker(\kl|_{W\times W})\neq \{0\}\), hence \(W\notin X\).

  By Lemma~\ref{lem:orbit-Xsa} there exist \(\theta \subset \Delta\) and \(h\in G\) such that \(W = \Ad(h)\cdot \mathfrak{r}_\theta\) and there exists \(h'\in G\) such that \(L = \Ad(h')\cdot \mathfrak{g}_{\chi_G}\).
Lemma~\ref{lem:ker_Xsa} implies that the intersection of \(\Ker(\kl|_{W\times W}) = \Ad(h) \cdot \mathfrak{u}_\theta \) and \(L\) is nontrivial.
We need to prove that \(L \subset \Ad(h) \cdot \mathfrak{u}_\theta\).

The element \(g= h^{-1} h'\) admits a Bruhat decomposition \(g = p \tilde{w} p'\), \ie \(p\) and \(p'\) are in the minimal parabolic subgroup \(P_\Delta\) and \(\tilde{w}\) belongs to \(N_K(\aaa)\) (see \eg~\cite[Th.\,7.40]{Knapp_LieGrp}). Let \(w\) be the class of \(\tilde{w}\) in~\(W = N_K(\aaa) / Z_K(\aaa)\).
Thus
\[ L = \Ad(h) \Ad(g) \cdot \mathfrak{g}_{\chi_G} = \Ad(h) \Ad(p) \Ad( \tilde{w}) \cdot \mathfrak{g}_{\chi_G} \]
(since \(\Ad(p')\cdot \mathfrak{g}_{\chi_G} = \mathfrak{g}_{\chi_G}\)) and \(L= \Ad(hp) \cdot \mathfrak{g}_{w\cdot \chi_G}\).
Also \(\Ad(h) \cdot \mathfrak{u}_\theta = \Ad(hp) \cdot \mathfrak{u}_\theta\).
Hence the Lie algebra \(\mathfrak{u}_\theta\) has a nontrivial intersection with \(\mathfrak{g}_{w\cdot \chi_G}\).
Since the root space decomposition is direct, this is possible if and only if \(\mathfrak{g}_{w\cdot \chi_G} \subset \mathfrak{u}_\theta\).
This implies that \(L\subset \Ad(h) \cdot \mathfrak{u}_\theta \subset W\).
\end{proof}

\begin{theorem}
  \label{thm:max-satak-semisimple}
  Let \(G\) be a real semisimple Lie group and let \(\Phi\subset \Delta\) be the set consiting of the simple restricted roots \(\alpha_{G'}\) given by Proposition~\ref{prop:G_makeP1} (see Table~\ref{table3}) for all the simple factors \(G'\) of~\(G\). 
  Let $\theta\subset\Delta$ be nonempty with \(\Phi\cap \theta \neq \emptyset\).
  For any word hyperbolic group $\Gamma$ and any \(P_\theta\)-Anosov representation \(\rho:\Gamma \to G\), the Riemannian locally symmetric space \(\rho(\Gamma) \backslash G/K\) admits a compactification modeled on the maximal Satake compactification of \(G/K\).
\end{theorem}

\begin{proof}[Proof of Theorem~\ref{thm:max-satak-semisimple}]
  There is a simple factor \(G'\) of \(G\) such that the projection \(\rho'\) of \(\rho\) to \(G'\) is \(P_{\{\alpha_{G'}\}}\)-Anosov.
The maximal Satake compactification \(\overX\) of \(X=G/K\) is the product \(\overX' \times \overX''\) of the maximal Satake compactification of \(G'/K'\) and of the maximal Satake compactification of the Riemannian symmetric space associated with the other factors of \(G\) (see Section~\ref{sec:satake-comp-1}). By Theorem~\ref{thm:max_Satake} there is an open set \(\Omega'\subset \overX'\) containing the Riemannian symmetric space of~\(G'\) such that the action of \(\Gamma\) via \(\rho\) on \(\Omega'\) is properly discontinuous and cocompact. It follows that the action of \(\Gamma\) via \(\rho\) on \(\Omega' \times \overX''\)  is properly discontinuous and cocompact.
\end{proof}

\subsection{Complex orthogonal groups}
\label{sec:compl-orth-groups}

Let \(\OO(b^{\CC})\) be the orthogonal group of a nondegenerate complex symmetric bilinear form \( b^{\CC}\) on a complex vector space \(V\) of dimension~\(n\).
 
 The parabolic subgroup \(P_1(b^{\CC})\) is defined as the stabilizer of the line \(\CC e_1\) where \(b^\CC(e_1, e_1)=0\); it is conjugate to its opposite.
The homogeneous space \(\mathcal{F}_1(b^{\CC}) = \OO(b^\CC)/P_1(b^\CC)\) is the set of \(b^\CC\)-isotropic complex lines in~\(V\).

The Riemannian symmetric space \(X_{b^\CC}\) of \(\OO(b^\CC)\) can be realized as a subset in the Grassmanian \(\mathrm{Gr}_{n}^{\RR}(V)\) of \(n\)-dimensional real subspaces of~\(V\), with compactification
\begin{equation} \label{eqn:overXbC}
  \overX_{b^\CC} = \{ W \in \mathrm{Gr}_{n}^{\RR}( V) \mid 
  b^{\CC}( W\times W)\subset \RR, \text{ and } 
  b^{\CC}(w,w) \leq 0 \, \forall w \in W\}.
\end{equation}
In fact \(\OO(b^\CC)\subset \OO( b)\) where \(b= \mathrm{Re}(b^\CC)\) and there is an inclusion \(X_{b^\CC}\subset X_{b}\). The compactification \(\overX_{b^\CC}\) is precisely the closure of \(X_{b^\CC}\) in \(\overX_{b}\).
It is isomorphic to a minimal Satake compactification if \(n\) is odd and to a generalized Satake compactification if \(n\) is even.

\begin{theorem}\label{thm:comp_loc_sym_spc_onC}
Let \(b^\CC\) be a nondegenerate complex symmetric bilinear form on a complex vector space \(V\) of dimension \(n\geq 3\), let $\Gamma$ be a word hyperbolic group, and let \(\rho:\Gamma\to \OO(b^\CC)\) be a \(P_1(b^{\CC})\)-Anosov representation with boundary map \(\xi: \partial_\infty \Gamma \to \mathcal{F}_1(b^{\CC}) \).
Let 
  \begin{equation*}
    \mathcal{N}_\rho = \bigcup_{\eta \in \partial_\infty \Gamma} \bigl\{ W \in \overX_{b^\CC} \mid \xi(\eta) \subset W \bigr\}.
  \end{equation*}
  Then the action of \(\Gamma\) on \(\Omega = \overX_{b^\CC} \smallsetminus \mathcal{N}_\rho\) is properly discontinuous and cocompact.
  The set \(\Omega\) contains the Riemannian symmetric space \(X_{b^\CC}\) and \(\rho(\Gamma)\backslash \Omega\) is a smooth compactification of \(\rho(\Gamma)\backslash X_{b^\CC}\).
\end{theorem}

\begin{proof}[Proof of Theorem~\ref{thm:comp_loc_sym_spc_onC}]
  Denote by \(\tau\) the natural injection of \(\OO(b^\CC)\) into \(\OO(b)\) where \(b= \mathrm{Re}( b^\CC)\). 
The representation \(\tau \circ \rho\) is \(P_2(b)\)-Anosov (see Proposition~\ref{prop:PthetaG_P1opq} and Lemma~\ref{lem:theta-comp-GL-Opq}). 

By Proposition~\ref{prop:propernessPk} the action of \(\Gamma\) on \(\overX_{b} \smallsetminus \mathcal{W}_{\tau \circ \rho}\) is properly discontinuous and thus the action of \(\Gamma\) via \(\rho\) on \(\overX_{b^\CC} \smallsetminus ( \mathcal{W}_{\tau\circ \rho}\cap \overX_{b^\CC})\) is as well properly discontinuous. 
Lemma~\ref{lem:nullW-C-space} implies that \(\mathcal{W}_{\tau\circ \rho}\cap \overX_{b^\CC} = \mathcal{N}_\rho\).

By Lemma~\ref{lem:cocompactness} the action of \(\Gamma\) on \(\overX_{b} \smallsetminus \mathcal{V}_{\tau \circ \rho}\) is cocompact and thus the action of \(\Gamma\) via \(\rho\) on \(\overX_{b^\CC} \smallsetminus ( \mathcal{V}_{\tau\circ \rho}\cap \overX_{b^\CC})\) is as well cocompact.
Since \(\mathcal{V}_{\tau\circ \rho}\cap \overX_{b^\CC} = \mathcal{N}_\rho\) the theorem follows.
\end{proof}

\begin{lemma}\label{lem:nullW-C-space}
  If \(W\in \overX_{b^\CC}\) then \(\Ker(b^{\CC}|_{W\times W})\) is a \(\CC\)-vector subspace of~\(\CC^n\).

  If \(L\in \mathcal{F}_1(b^{\CC})\), then \(L\cap W\neq \{0\}\) \(\Leftrightarrow\) \(L\subset W\).
\end{lemma}

\begin{proof}[Proof of Lemma~\ref{lem:nullW-C-space}]
  The kernel \(\Ker(b^{\CC}|_{W\times W})\) is a real vector space, we need to prove that it is stable by multiplication by \(\sqrt{-1}\). 
Let \(z\in \Ker(b^{\CC}|_{W\times W})\), then \(b^{\CC}(z,z)=0\) and \(b^{\CC}(x,z)= 0\) for all \(x\in W\). 
Set \(y= \sqrt{-1} z \) and \(b = \mathrm{Re}(b^{\CC})\).
One has \(b(y,y)= 0 \) and \(b(x,y)=0\) for all \(x\in W\). 
By Lemma~\ref{lem:incid_iso_neg}, \(y\in W\) and \(y\) belongs to \(\Ker(b^{\CC}|_{W\times W})\).

If \(L\cap W \neq \{0\}\) then \(L \cap \Ker(b^{\CC}|_{W\times W})\neq \{0\}\) and \(L\subset \Ker(b^{\CC}|_{W\times W})\subset W\) since \(L\) has complex dimension~\(1\).
\end{proof}

\appendix

\section{Satake compactifications}
\label{sec:satake-comp}

\subsection{Satake compactifications}\label{sec:satake-comp-1}

In this section we briefly review the construction of the Satake compactification of a Riemannian symmetric space $X = G/K$, which was originally defined in \cite{Satake_annals}.
We denote by $\mathcal{H}_n$ the space of Hermitian $(n\times n)$ matrices over~$\CC$.

Let $\tau: G \to \PSL(n,\CC)$ be an irreducible projective representation with finite kernel.
We may assume that \(\tau(K)\subset \PSU(n)\). 
By definition, the \emph{Satake compactification} $\overX_\tau$ of $X$ associated with~$\tau$ is the closure in $\PP(\mathcal{H}_n)$ of the image of $X$ under the embedding \(X \rightarrow \PP(\mathcal{H}_n)\) given by $gK \mapsto \RR({\tau}(g) {\tau}(g)^*)$, where \(M^*\) is the transpose-conjugate of a matrix \(M\).

The structure of the Satake compactification $\overX_\tau$ as a $G$-space only depends on the support
\[ \theta_\tau = \{ \alpha \in \Delta \mid ( \chi_\tau , \alpha) >0\} \]
of the highest weight \(\chi_\tau\) of the irreducible representation~$\tau$.
Satake compactifications have the following properties:
\begin{enumerate}
  \item The compactification $\overX_{\tau}$ has finitely many $G$-orbits, including a unique open $G$-orbit, namely $X= G/K$, and a unique closed orbit, which identifies with $G/P_{\theta_\tau}$.
  \item \label{item:Satake-surj} If $\theta_{\tau'} \subset \theta_\tau$ then there exists a continuous (hence proper) surjective \(G\)-equivariant map \(\pi_{\tau, \tau'}: \overX_\tau \to \overX_{\tau'}.\)
\item Every Satake compactification of a product is a product of Satake compactifications.
\end{enumerate}
By \eqref{item:Satake-surj}, the Satake compactification $\overX_{\tau}$ for \(\theta_\tau = \Delta\) surjects onto any Satake compactification $\overX_{\tau'}$ of~$X$; it is called the \emph{maximal} Satake compactification of~$X$.
On the other hand, Satake compactifications of the form $\overX_{\tau}$ for \(\sharp \theta_\tau =1\) are called \emph{minimal} Satake compactifications.
The maximal Satake compactification of~$X$ is a manifold with corners \cite[Prop.\,I.19.27]{Borel_Ji}.
The set \(\theta_\tau\) will be called the \emph{support} of the Satake compactification~\(\overX_\tau\).

\subsection{Orbits description}
\label{sec:orbits-description}

The finitely many orbits of a Satake compactification are described by the some combinatorial data associated with the irreducible representation \(\tau\).
As our convention for parabolic groups is opposite to~\cite{Borel_Ji}, the terminology and description of the orbits and the stabilizers has to be adapted from the classical case.

For any subset \(\theta \subset \Delta\) the parabolic algebra \(\mathfrak{p}_\theta\) is the direct sum
\begin{equation*}
  \mathfrak{p}_\theta = \mathfrak{u}_\theta \oplus \aaa_\theta \oplus \mathfrak{m}_\theta
\end{equation*}
where \(\aaa_\theta = \bigcap_{\alpha \in \Delta \smallsetminus\theta} \Ker(\alpha)\), \(\mathfrak{m}_\theta = \mathfrak{z}_{\mathfrak{k}}(\aaa) \oplus \aaa^\theta \oplus \bigoplus_{\alpha \in \Sigma \cap \mathrm{span}( \Delta \smallsetminus \theta)} \mathfrak{g}_\alpha\), and \(\aaa^\theta = \aaa \cap (\aaa_\theta)^{\perp_{\kl}}\).
The corresponding Lie groups are denoted by \(U_\theta\), \(A_\theta\) and \(M_\theta\).
The map \(U_\theta \times A_\theta \times M_\theta \to P_\theta \mid (u,a,m)\mapsto uam\) is a diffeomorphism and we will simply write in the sequel \(P_\theta = U_\theta A_\theta M_\theta\).

A subset \(\theta \subset \Delta\) will be said \emph{\(\tau\)-admissible} if the graph with vertex sets \((\Delta \smallsetminus \theta) \cup \{ \chi_\tau\}\) and edges between every pairs with a non zero scalar products is connected.
(One usually says that \(\Delta\smallsetminus \theta\) is \(\tau\)-connected.)

For such a subset let \(\theta^\vee = \{ \alpha \in \Delta \mid \exists \beta \in (\Delta \smallsetminus \theta) \cup \{\chi_\tau\},\, (\alpha, \beta) \neq 0\}\) be the set of the elements of \(\Delta\) being non-orthogonal to an element in  \((\Delta \smallsetminus \theta) \cup \{\chi_\tau\}\).
The set \(\theta^{\ddagger} = \theta \cap \theta^\vee\) is called the \emph{\(\tau\)-nucleus} of \(\theta\).
(The usual terminology says that \(\Delta \smallsetminus \theta^\cap\) is the \(\tau\)-saturation of \(\Delta\).)
Note that \(M_{\theta^{\ddagger}}\) is the almost product of \(M_\theta\) and \(M_{\theta^\vee}\) and that these two last groups commute.

The Satake compactification \(\overX_\tau\) admits the following description:
\begin{asparaenum}[(a)]
  \item it is the disjoint union of the \(G\)-orbits of points \(x_{\theta}\) over the \(\tau\)-admissible sets \(\theta\subset \Delta\);
  \item the stabilizer of \(x_{\theta}\) is the product \(U_{\theta^{\ddagger}} A_{\theta^{\ddagger}} M_{\theta^\vee} (K \cap M_\theta)\); in particular is contained in \(P_{\theta^{\ddagger}}\);
  \item the orbit \(G \cdot x_{ \theta}\) fibers over the flag manifold \(G/ P_{\theta^{\ddagger}} = \mathcal{F}_{\theta^{\ddagger}}\) and the fibers are isomorphic to \(M_\theta / (K\cap M_\theta)\), \ie to the Riemannian symmetric space associated with the reductive group~\(M_\theta\);
  \item the orbit \(G\cdot x_{\emptyset}\) is the copy of the Riemannian symmetric space \(G/K\);
  \item the unique closed orbit is \(G\cdot x_{ \Delta}\).
\end{asparaenum}
In order to describe the topology on \(\overX_\tau\), it is enough to understand the closure of the Weyl chamber.
Let \((H_n)_{n\in \NN}\) be a sequence in \(\overline{\aaa}^+\), then the sequence \((\exp( H_n)\cdot x_{\emptyset})_{n\in \NN}\) converges in \(\overX_\tau\) if and only if there exists a \(\tau\)-admissible set \(\theta\subset \Delta\) such that
\begin{enumerate}[(i)]
  \item For each \(\alpha\in \Delta \smallsetminus \theta\) the sequence \((\langle{\alpha, H_j}\rangle)_{n\in \NN}\) converges to some \(t_\alpha\in \RR\) and
  \item for every \(\tau\)-admissible set \(\theta' \subsetneq \theta\) there exists \(\alpha \in \theta \smallsetminus \theta'\) such that \(\lim_{n \to \infty} \langle\alpha, H_n\rangle = +\infty\).
\end{enumerate}
Furthermore, if \(H\) is the unique element of \( \aaa^\theta\) such that \(\langle \alpha, H\rangle = t_\alpha\) for all \(\alpha\in \Delta \smallsetminus \theta\), then \( \lim_{n\to \infty}\exp( H_n)\cdot x_{\emptyset} = \exp(H)\cdot x_{ \theta}\).

\subsection{A minimal Satake compactification of $X_{b}$} 
\label{sec:compactification_Xpq}

Recall that for real symmetric bilinear forms $b$, a compactification $\overX_b$ of the Riemannian symmetric space $X_b$ of $\OO(b)$ was constructed in Section~\ref{sec:comp-anos-locally}.
Similarly, for complex symmetric bilinear forms $b^{\CC}$, a compactification $\overX_{b^{\CC}}$ of the Riemannian symmetric space $X_{b^{\CC}}$ of $\OO(b^{\CC})$ was constructed in Section~\ref{sec:compl-orth-groups}.

\begin{proposition} \label{prop:overXb-min-Satake}
Let \(b\) be a nondegenerate symmetric bilinear form of signature \((p,q)\) on a real vector space~\(V\).
If $p>q$, then $\overX_{b}$ is a minimal Satake compactification of~$X_{b}$.
\end{proposition}

When \(p=q\), \(\overX_b\) is not a Satake compactification, see Example~\ref{exam:Xb_not_Sata}.

\begin{proof}
The natural representation $\tau : \OO(b)\to\GL_{\RR}(\bigwedge^q V)$ identifies $X_{b}$ with a subset of $\PP(\bigwedge^q V)$, and $\overX_{b}$ with the closure of this set in $\PP(\bigwedge^q V)$.
Since \(\bigwedge^q V\) is an irreducible representation of \(\OO(b)\), the result of \cite{Koranyi_CartanHelgason} applies to show that \(\overX_{b}\) is a Satake compactification of~$X_{b}$ and that its support is \(\{ \alpha \in \Delta \mid (\chi_\tau, \alpha ) \neq 0\}\) if \(\chi_\tau\) is the highest weight of \(\tau\).
This support is thus \(\{ \alpha_q\}\) and the compactification is a minimal Satake compactification.
\end{proof}

\begin{proposition}
  Let \(b^\CC\) be a nondegenerate symmetric complex bilinear form on a complex vector space \(V\) of dimension~\(n\).
  If \(n\) is odd, then \(\overX_{b^\CC}\) is a minimal Satake compactification of~$X_{b^\CC}$.
\end{proposition}

\begin{proof}
  Write \(n=2m+1\).
  There is a basis \((e_1, \dots, e_n)\) of \(V\) such that \(b^\CC(x,y) = \sum_{i=1}^{n} x_i y_{n-i+1}\) for every \(x = \sum_{i=1}^{n} x_i e_i\) and  \(y = \sum_{i=1}^{n} y_i e_i\) of \(V\).
One can take \(K = \U(n)\) and \(\aaa\) the diagonal matrices in that basis.
Namely
\begin{equation*}
  \aaa = \{ \mathrm{diag}( \lambda_1, \dots, \lambda_m, 0 , -\lambda_m, \dots, -\lambda_1) \mid \lambda_1, \dots, \lambda_m \in \RR\}.
\end{equation*}
For \(i=1, \dots, m\), set \(\langle \varepsilon_i, \mathrm{diag}( \lambda_1, \dots, \lambda_m, 0 , -\lambda_m, \dots, -\lambda_1) \rangle = \lambda_i\).
The restricted root system is
\[ \Sigma = \{\pm \varepsilon_i \pm \varepsilon_j \mid 1\leq i < j \leq m\} \cup \{ \pm \varepsilon_i \mid 1\leq i \leq m\}. \]
A system of simple restricted roots is \(\Delta =\{ \alpha_1, \dots, \alpha_m\}\) where \(\alpha_i = \varepsilon_i - \varepsilon_{i+1}\) for \(1\leq i\leq m-1\) and \(\alpha_m = \varepsilon_m\). 
We will show that \(\overX_{b^\CC}\) satisfies the above axiomatic description of Satake compactification whose support is \( \{\alpha_m\}\).
The admissible sets are \( \theta_0 = \emptyset\), \(\theta_1= \{ \alpha_1\}\), \dots{}, \(\theta_m =\Delta\). One checks that \(\theta_{0}^{\vee} = \Delta\), \(\theta_{0}^{\ddagger} = \emptyset\) and \(\theta_{i}^{\vee} = \{ \alpha_i, \dots, \alpha_m\}\), \(\theta_{i}^{\ddagger} = \{\alpha_i\}\) for \(1\leq i \leq m\).

  We first consider the orbits description of \(\overX_{b^\CC}\). By Lemma~\ref{lem:nullW-C-space}, for any \(W\in \overX_{n,\CC}\), the set \(\Ker(b^{\CC}|_{W\times W})\) is a \(\CC\)-vector subspace of \(V\). Applying Witt's theorem, one shows that the sets
  \[ \mathcal{U}_i = \{ W \in \overX_{b^\CC} \mid \dim_\CC \Ker(b^{\CC}|_{W\times W})=i\}, \]
  \(i=0, \dots,m\), are the \(\OO(b^\CC)\)-orbits in \(\overX_{b^\CC}\).
For \(0\leq i \leq m\), the real vector space
  \begin{multline*}
    W_{\theta_i} = \mathrm{Span}_\RR( e_1 , \sqrt{-1} e_1, \dots , e_i, \sqrt{-1} e_i, e_{i+1} - e_{n-i}, \sqrt{-1} (e_{i+1}+ e_{n-i}),\\ \dots, e_{m} - e_{m+2}, \sqrt{-1} (e_{m}+ e_{m+2}), \sqrt{-1}e_{m+1})
  \end{multline*}
  belongs to \(\mathcal{U}_i\) and its stabilizer in \(\OO(b^\CC)\) is precisely the group \(U_{\theta_{i}^{\ddagger}} A_{\theta_{i}^{\ddagger}} M_{\theta_{i}^\vee} (K \cap M_{\theta_{i}})\) described above and contained in the parabolic subgroup \(P_{\theta_{i}^{\ddagger}}\).

For an element \(H\) in \(\aaa\) one has
  \begin{multline*}
    \exp(H) \cdot W_{\theta_0} = \mathrm{Span}_\RR(  e_{1} - e^{-2\langle \varepsilon_1, H \rangle}e_{n}, \sqrt{-1} (e_{1}+ e^{-2\langle \varepsilon_1, H \rangle} e_{n}),\\ \dots, e_{m} - e^{-2\langle \varepsilon_m, H \rangle} e_{m+2}, \sqrt{-1} (e_{m}+ e^{-2\langle \varepsilon_m, H \rangle} e_{m+2}), \sqrt{-1}e_{m+1}).
  \end{multline*}
From this for a sequence \((H_k)_{k\in\NN}\) in \((\overline{\aaa}^+)^\NN\), the sequence \((\exp(H_k)\cdot W_{\theta_0})_{k\in\NN}\) converges if and only if there exists \(0\leq i \leq m\) for which \((\langle \varepsilon_j, H_k\rangle)_{k\in \NN}\) goes to infinity for \(j\leq i\) and has a limit in \(\RR\) for \(j>i\). This is equivalent to \((\langle \alpha_j, H_k\rangle)_{k\in \NN}\) goes to infinity for \(j\leq i\) and has a limit in \(\RR\) for \(j>i\). It follows that the closure of the Weyl chamber satisfies the above description.
\end{proof}

\begin{remark}
  For even~\(n\), the same analysis can be used to show that \(\overX_{b^\CC}\) is \emph{not} a Satake compactification.
\end{remark}

\subsection{Generalized Satake compactifications}
\label{sec:generalizedSatake}

The classical notion of Satake compactification does not behave well with respect to totally geodesic subspaces: 
If $\overX_\tau$ is a Satake compactification of $X$ and if $Y\subset X$ is a totally geodesic subsymmetric space, then the closure of $Y$ in $\overX_\theta$ is not always a Satake compactification of~$Y$. 

\begin{example}\label{exam:Xb_not_Sata}
Consider a real vector space \(V'\) equipped with a symmetric bilinear form \(b'\) of signature \((p, p+1)\) and let \(V\subset V'\) be a subspace such that \(b = b'|_{V\times V}\) is of signature \((p,p)\).
Then the closure of $X_{b}$ inside $\overX_{b'}$ is not a Satake compactification. Indeed this closure is \(\overX_b\) which contains
 two closed $\OO(b)_0$-orbits, contradicting the above description of Section~\ref{sec:orbits-description}.
\end{example}

In order to obtain a class of compactifications which have this functoriality properties we will have to consider a small generalization of Satake compactifications, which we call \emph{generalized Satake compactifications}.
The only difference is that we allow the representation $\tau$ to be a sum of irreducible representations.

\begin{remark}
Compare with \cite[\S\,5.3]{Satake_annals}, where Satake considers reducible representations, but then takes the closure in $\PP(\mathcal{H}_{n_1} ) \times \cdots \times \PP(\mathcal{H}_{n_k} )$ instead of $\PP(\mathcal{H}_{\sum_{i=1}^{k} n_i})$. 
\end{remark}

\begin{definition} \label{defi:gen-Satake}
Let $G$ be a semisimple Lie group and $\tau: G \to \SL(n,\CC)$ a faithful projective representation with \(\tau(K)\subset \PSU(n)\).
The \emph{generalized Satake compactification} of $X=G/K$ associated with~$\tau$ is the closure of the image of $X$ under the map \(X \rightarrow \PP(\mathcal{H}_n)\) given by $gK \mapsto \RR({\tau}(g) {\tau}(g)^*)$. 
\end{definition}

\begin{lemma} \label{lem:genSatake-subsymm}
Let $X=L/K$ be a Riemannian symmetric space, $\overX$ a generalized Satake compactification of~$X$, and $Y=H/(K\cap H)\hookrightarrow X$ a totally geodesic subsymmetric space of~$X$.
Then the closure of $Y$ in $\overX$ is a generalized Satake compactification of~$Y$.
\end{lemma}

\begin{proof}[Proof of Lemma~\ref{lem:genSatake-subsymm}]
  Let $\tau: L \to \SL(n,\CC)$ be a representation with finite kernel such that \(\overX = \overX_\tau\) and let $\phi: H \to L$ be the Lie group homomorphism associated with the embedding \(Y \hookrightarrow X\).
Then the closure of $Y$  is the generalized Satake compactification associated with $ \tau \circ \phi: H \to \SL(n,\CC)$.
\end{proof}

From this we deduce the following.

\begin{lemma}\label{lem:gener-satake-comp-product}
  Let \(X=G/K\) be a Riemannian symmetric space, \(\tau_1 : G \to \SL(n_1, \CC)\) a representation with \(\tau_1 (K)\subset \U(n_1)\) and \(\tau_2: G \to \SL(n_2, \CC)\) a representation with finite kernel and with \(\tau_2(K)\subset \U(n_2)\).
Let \(\overX_i\), \(i=1,2\) be the closure of \(\tau_i(G)/ \tau_i(K)\) in \(\PP( \mathcal{H}_{n_i})\) so that \(\overX_2\) is a generalized Satake compactification. 
Let 
\begin{align*}
\psi: X & \longrightarrow \overX_1 \times \overX_2 \\
 g\cdot K & \longmapsto (\RR ( \tau_1(g) \tau_{1}^{*}(g)), \RR ( \tau_2(g) \tau_{2}^{*}(g))).
\end{align*}

Then the closure of \(\psi(X)\) is a generalized Satake compactification.
\end{lemma}

\begin{proof}[Proof of Lemma~\ref{lem:gener-satake-comp-product}]
  Apply Lemma~\ref{lem:genSatake-subsymm} with \(L= \SL(n_1, \CC) \times \SL(n_2, \CC)\) and \(H = (\tau_1, \tau_2)(G)\).
\end{proof}

We say that a compactification \(\overX_1\) \emph{dominates} a compactification \(\overX_2\) if there is a continuous \(G\)-equivariant map \(\overX_1 \to \overX_2\), such a map is necessarly surjective and proper.

\begin{lemma}
  \label{lem:gal-sata-corners}
  Let \(\overX\) be a generalized Satake compactification dominating the maximal Satake compactification, then \(\overX\) is a manifold with corners.
\end{lemma}

\begin{proof}[Proof of Lemma~\ref{lem:gal-sata-corners}]
  Let \(\psi: \overX \to \overX^{max}\) be the continuous \(G\)-equivariant map.
  It is easily seen that the closure \(\overline{F}\) of \(F\) the \(\exp(\aaa)\)-orbit of the base point \(x_0\) in \(G/K\subset \overX\) is a manifold with corners.
  Let \(\overline{F}^{+} \subset \overline{F}\) be the closure of \(\exp( \overline{\aaa}^+)\cdot x_0\).
  Let \(x\) be any point of \(\overX\).
  Using the Cartan decomposition of \(G\) one can assume that \(x\in \overline{F}^+\).
  Furthermore, using the Iwasawa decomposition, one deduces that the map \(f:U^{-}_{\emptyset} \times \overline{F} \to \overX\) is surjective in a neighborhood of \((e,x)\) into a neighborhood of \(x\).
  Since the corresponding map \(U^{-}_{\emptyset} \times \psi( \overline{F}) \to \overX^{max}\) is a local diffeomorphism we deduce that \(f\) is also injective.
  As \(U^{-}_{\emptyset} \times \overline{F}\) is a manifold with corners, the lemma follows.
\end{proof}


\providecommand{\bysame}{\leavevmode\hbox to3em{\hrulefill}\thinspace}

\end{document}